\documentclass[reqno,12pt]{amsart}
\usepackage{a4wide}
\usepackage{amsmath}
\usepackage{amssymb}
\usepackage{amsthm}
\usepackage{amstext}
\usepackage{geometry}
\usepackage{fancyhdr}
\usepackage{array}
\usepackage{graphicx} 
\usepackage{hyperref}
\usepackage{fullpage} 
\usepackage{bm} 
\allowdisplaybreaks
\hypersetup{ 
colorlinks=true,
linkcolor=blue, 
urlcolor=red,
citecolor=green,
} 
\usepackage{tikz}
\usetikzlibrary{
 shapes.geometric,
 arrows.meta,
 angles,
 quotes,
 patterns,
 arrows
}
\usetikzlibrary{angles, quotes, calc}

\usepackage{subcaption}
\allowdisplaybreaks[4]

\newtheorem{theorem}{Theorem}[section]

\newtheorem{corollary}[theorem]{Corollary}

\newtheorem{lemma}[theorem]{Lemma}
\newtheorem{proposition}[theorem]{Proposition} 
\newtheorem{remark}[theorem]{Remark}

\numberwithin{equation}{section}

\begin{document}

\title{On the uniqueness of the critical point of $\psi_\Omega$}

\author{Junyuan Liu}
\address{[Junyuan Liu] College of Informatics, Huazhong Agricultural University
Wuhan 430070, Hubei, P. R. China}
\email{jyliu@mail.hzau.edu.cn}

\author{Shuangjie Peng}
\address{[Shuangjie Peng] School of Mathematics and Statistics, Key Laboratory of Nonlinear Analysis and Applications (Ministry of Education), Central China Normal University, Wuhan, 430079, P. R. China}
\email{sjpeng@mail.ccnu.edu.cn}

\author{Fulin Zhong$^{\dagger}$}
\address{[Fulin Zhong] School of Mathematics and Statistics, Central China Normal University, Wuhan, Hubei 430079, P. R. China}
\email{flzhong@mails.ccnu.edu.cn}

\thanks{$^{\dagger}$ Corresponding author.}
\date{\today}

\begin{abstract} 
We prove that for any bounded convex domain $\Omega \subset \mathbb{R}^n$, the function 
\begin{equation*}
\psi_\Omega(\xi) = \int_{\mathbb{R}^n\setminus\Omega} \frac{\mathrm{d}x}{|x-\xi|^{2n}}, \quad \xi\in\Omega,
\end{equation*} 
has exactly one critical point. This confirms a conjecture proposed by Clapp, Pistoia and Salda\~na in [J. Math. Pures Appl. 205 (2026), 103783]. The proof uses a spherical coordinates representation to write $\psi_\Omega$ as an integral of the distance function $\rho(\xi,\omega)$. This approach is not limited to $\psi_\Omega$. Instead, it provides a general framework for analyzing a broad class of functionals involving the boundary distance. 
We also examine non-convex domains. In particular, a single annulus exhibits a full circle of critical points, while multiple concentric annuli produce finitely many critical spheres. These examples show that the convexity hypothesis is essential for the uniqueness conclusion. The method developed here for handling spherical integrals involving the distance function is likely to be useful in other geometric and analytic contexts. 
\end{abstract}

\maketitle



\vspace{0.25cm}


\vspace{0.25cm}

\section{Introduction}
The study of critical points of solutions to elliptic equations has a long and rich history, intimately connected with the geometry of the underlying domain. For positive solutions of semilinear problems such as
\begin{equation*}
\begin{cases}
-\Delta u = f(u) &\text{in }\Omega,\\ 
u=0 &\text{on }\partial\Omega,
\end{cases}
\end{equation*}
the number and location of critical points are known to be strongly influenced by the convexity or non‑convexity of $\Omega$.

The relationship between the geometry of the domain and the number of critical points of solutions to elliptic equations has been extensively investigated. For convex domains, Gidas-Ni-Nirenberg \cite{GidasNiNirenberg1979} proved uniqueness under symmetry assumptions. Without symmetry the question remains open, but important partial results are known for specific nonlinearities. The torsion problem was solved by Makar‑Limanov \cite{MakarLimanov1971}, while uniqueness for the first eigenfunction follows from the log‑concavity results of Brascamp-Lieb \cite{BrascampLieb1976} and Acker-Payne-Philippin \cite{AckerPaynePhilippin1981}, later extended to higher dimensions by Korevaar-Lewis \cite{KorevaarLewis1987}. For semi‑stable solutions in planar convex domains, Cabr\'e-Chanillo \cite{CabreChanillo1998} established uniqueness under positive boundary curvature, a result later extended by De Regibus-Grossi-Mukherjee \cite{DeRegibusGrossiMukherjee2021} to domains where curvature may vanish. More recently, Battaglia-De Regibus-Grossi \cite{BattagliaDeRegibusGrossi2024} gave geometric conditions guaranteeing uniqueness for the Poisson problem with variable source term, even for non‑convex domains. In striking contrast, non‑convex domains can support arbitrarily many critical points: Gladiali-Grossi \cite{GladialiGrossi2022} constructed ``almost convex'' planar domains whose torsion solutions have arbitrarily many maxima. Domains with holes also exhibit multiple critical points: a topological lower bound follows from Lusternik–Schnirelmann category, and Grossi–Luo \cite{GrossiLuo2020} gave a precise asymptotic analysis for the semilinear problem with a small hole. For further related results, see also  \cite{GladialiGrossi2024,GladialiGrossiLuoYan2025,GladialiGrossiLuoYan2026}.

A particularly interesting question arises from a recent work by Clapp, Pistoia and Salda\~na \cite{ClappPistoiaSaldana2026}, who studied the semilinear elliptic problem
\begin{equation*}
-\Delta u = Q_\Omega(x) |u|^{p-1-\varepsilon}u, \quad u \in D^{1,2}(\mathbb{R}^n),
\end{equation*}
where $p = \frac{n+2}{n-2}$ is the critical Sobolev exponent, $n \ge 3$, and
\begin{equation*}
Q_\Omega(x) = \mathbf{1}_\Omega(x) - \mathbf{1}_{\mathbb{R}^n\setminus\Omega}(x) =
\begin{cases}
+1, & x \in \Omega, \\
-1, & x \in \mathbb{R}^n \setminus \overline{\Omega}.
\end{cases}
\end{equation*}
Thus $Q_\Omega$ is a discontinuous coefficient that changes sign across the boundary $\partial\Omega$. The space $D^{1,2}(\mathbb{R}^n)$ denotes the homogeneous Sobolev space defined as the completion of $C_c^\infty(\mathbb{R}^n)$ with respect to the norm $\|u\|_{D^{1,2}(\mathbb{R}^n)} := \|\nabla u\|_{L^2(\mathbb{R}^n)}$. The sign-changing nature of $Q_\Omega$ is relevant in several physical applications. For instance, it describes optical waveguides propagating through a stratified dielectric medium \cite{Stuart1990, Stuart1993}.

Using a Lyapunov-Schmidt reduction, the authors of \cite{ClappPistoiaSaldana2026} showed that for sufficiently small $\varepsilon>0$, positive solutions concentrating at a single point exist, and their concentration profiles are given by standard bubbles. They demonstrated that positive concentrating solutions exist and that their concentration points are linked to the critical points of a function $\psi_\Omega$ defined on $\Omega$ by
\begin{equation*}
\psi_\Omega(\xi) = \int_{\mathbb{R}^n\setminus\Omega} \frac{1}{|x-\xi|^{2n}} \mathrm{d}x, \quad \xi \in \Omega.
\end{equation*}
In particular, the number of positive concentrating solutions is directly linked to the number of non-degenerate critical points of $\psi_\Omega$. In view of this, the authors raised in \cite[Question 1.6]{ClappPistoiaSaldana2026} the following natural \textbf{conjecture}: 
\begin{equation*}
\emph{If $\Omega$ is convex, does $\psi_\Omega$ have a unique critical point?}
\end{equation*}
Note that even for such a special case as the ball, the authors did not obtain uniqueness.

The convexity assumption is both natural and restrictive: convex domains appear frequently in applications, and they often simplify geometric questions. However, the function $\psi_\Omega$ is not obviously convex, nor is it obvious that it should have a unique critical point. Indeed, non-convex domains, such as dumbbell shapes, may admit multiple critical points. This phenomenon discussed in \cite{ClappPistoiaSaldana2026} is illustrated in Figure \ref{fig:comparison}. The conjecture thus asks whether convexity forces uniqueness, a question that is geometrically appealing and analytically non-trivial.

\begin{figure}[htbp]
\centering
\begin{tikzpicture} 
\begin{scope}[xshift=5cm]
\draw[thick] 
 (0,0) circle (1.2)
 (5,0) circle (1.2);
 
 \draw[thick] 
 (1.2,0.05) -- (3.8,0.05) -- (3.8,-0.05) -- (1.2,-0.05) -- cycle;
\node[circle, fill=red, inner sep=1pt, label={below:min$_1$}] at (-0.5,0) {};
\node[circle, fill=red, inner sep=1pt, label={below:min$_2$}] at (5.5,0) {};
\node at (2.5,1) {$\Omega_{\text{dumbbell}}$};
\node at (2.5,-1.8) {Multiple critical points};
\end{scope}
\end{tikzpicture}
\caption{A non-convex domain may exhibit multiple critical points.}
\label{fig:comparison}
\end{figure}

In this paper, we give an \textbf{affirmative} answer to this question. We prove that for any bounded convex domain $\Omega\subset\mathbb{R}^n$, the function $\psi_\Omega$ is strictly convex on $\Omega$, and therefore possesses exactly one critical point, which is necessarily its global minimum. The proof reveals a hidden convexity that is not apparent from the original definition of $\psi_\Omega$ as a volume integral. It becomes visible only after a suitable spherical reduction that brings the geometry of the boundary into play. Our proof does not rely on symmetry assumptions and works for every bounded convex domain in any dimension. We also discuss several examples that illustrate the drastic change in the number of critical points when convexity is lost, showing that the conjecture is sharp and that a simple relation with topological invariants cannot be expected in general. 
Moreover, the key properties of the distance function $\rho(\xi,\omega)$ developed in this paper (see Appendix \ref{section-appendix}), including its regularity, uniform bounds, and the explicit structure of its gradient, are not limited to the specific functional $\psi_\Omega$. They provide a flexible framework for analyzing a broad class of functionals that admit a spherical representation involving the distance to the boundary, and are expected to have further applications in convex analysis, potential theory, and related fields.

The fact that $\psi_\Omega(\xi)\to+\infty$ as $\xi\to\partial\Omega$ implies that $\psi_\Omega$ attains a minimum in $\Omega$, and this minimizer is a critical point. Moreover, we have the following theorem.

\begin{theorem} \label{thm:main}
Let $\Omega \subset \mathbb{R}^n$ be a bounded convex domain (no further regularity is required). Then the function
\begin{equation*}
\psi_\Omega(\xi) = \int_{\mathbb{R}^n\setminus\Omega} \frac{\mathrm{d}x}{|x-\xi|^{2n}}, \quad \xi \in \Omega,
\end{equation*}
has a unique critical point in $\Omega$. This critical point is the global minimizer of $\psi_\Omega$.

Moreover, if in addition $\Omega$ is symmetric with respect to some point $p\in\Omega$ (for example, $\Omega$ is a ball or an ellipsoid centered at $p$), then $p$ is the unique critical point.
\end{theorem}

\begin{figure}[htbp]
\centering
\begin{tikzpicture}[scale=1.1]

\begin{scope}[xshift=-3cm]
\draw[thick] (90:1.2) -- (210:1.2) -- (330:1.2) -- cycle;
\node[circle, fill=red, inner sep=1pt, label={below:min}] at (0,0) {};
\node at (0,1.5) {$\Omega_{\text{triangle}}$};
\node at (0,-1.5) {Equilateral triangle};
\end{scope}

\begin{scope}[xshift=3cm]
\draw[thick] 
 (-1.0,1.0) -- (1.0,1.0) 
 arc[start angle=90, end angle=-90, radius=1.0] 
 -- (-1.0,-1.0) 
 arc[start angle=270, end angle=90, radius=1.0] -- cycle;
\node[circle, fill=red, inner sep=1pt, label={below:min}] at (0,0) {};
\node at (0,1.5) {$\Omega_{\text{stadium}}$};
\node at (0,-1.5) {Stadium};
\end{scope}
\end{tikzpicture}
\caption{Two convex domains with symmetry.}
\label{fig:weird_symmetric}
\end{figure}
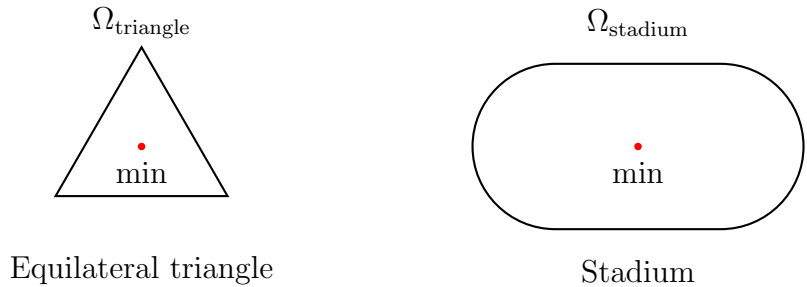

As a concrete example, Figure \ref{fig:weird_symmetric} depicts a convex domain with rotational symmetry, and its unique critical point is located at the symmetry center, in accordance with Theorem \ref{thm:main}. Furthermore, we obtain the following result for the original semilinear elliptic problem studied in \cite{ClappPistoiaSaldana2026}.

\begin{corollary}
Let $\Omega\subset\mathbb{R}^n$ be a bounded convex and smooth domain. Then there exists $\varepsilon_0>0$ such that, for all $\varepsilon \in (0,\varepsilon_0)$, the slightly subcritical problem
\begin{equation*}
-\Delta u = Q_\Omega(x) |u|^{p-1-\varepsilon}u,\quad u\in D^{1,2}(\mathbb{R}^n) \cap L^{p+1-\varepsilon}(\mathbb{R}^n),
\end{equation*}
admits a positive single‑bubbling solution $u_\varepsilon$, and its concentration point is the unique critical point of $\psi_\Omega$.
\end{corollary}

\begin{remark}
The geometry of the domain strongly influences the number of critical points of $\psi_\Omega$, see Figure \ref{fig:nonconvex}. The square with a hole (Figure \ref{fig:hole}) is a typical example where the presence of a topological obstruction forces the existence of at least four critical points. The annular domain (Figure \ref{fig:annulus}) is even more striking: its rotational symmetry leads to a whole continuum of critical points, a full circle. These examples highlight the deep connection between the topology of $\Omega$ and the critical point theory of $\psi_\Omega$, as already hinted at by the Lusternik-Schnirelmann estimates in \cite{ClappPistoiaSaldana2026}.

Also, the example in Figure \ref{fig:annulus} illustrates that a precise formula linking the number of critical points of $\psi_\Omega$ to fine topological invariants (such as the genus or the Betti numbers) cannot be expected in general. Hence, the critical point structure of $\psi_\Omega$ is a subtle problem where symmetry and convexity play a decisive role.
\end{remark}

\begin{figure}[htbp]
\centering
\begin{subfigure}[b]{0.45\textwidth}
\centering
\begin{tikzpicture}[scale=0.6]

\fill[white] (-2,-2) rectangle (2,2);

\fill[white] (0,0) circle (0.8);

\draw[thick] (-2,-2) rectangle (2,2);

\draw[thick] (0,0) circle (0.8);

\node at (0,2.4) {$\Omega_{\text{hole}}$};

\fill[red] (-1.5,0) circle (2pt) node[below] {};
\fill[red] (1.5,0) circle (2pt) node[below] {};
\fill[red] (0,1.5) circle (2pt) node[right] {};
\fill[red] (0,-1.5) circle (2pt) node[left] {};
\end{tikzpicture}
\caption{Square with a hole (non-convex).}
\label{fig:hole}
\end{subfigure}
\hfill
\begin{subfigure}[b]{0.45\textwidth}
\centering
\begin{tikzpicture}[scale=0.7]

\fill[white] (0,0) circle (2) (0,0) circle (1);
\draw[thick] (0,0) circle (2);
\draw[thick] (0,0) circle (1);

\node at (0,2.4) {$\Omega_{\text{annular}}$};

\draw[dashed, red] (0,0) circle (1.5);
\end{tikzpicture}
\caption{Annular domain (non-convex).}
\label{fig:annulus}
\end{subfigure}
\caption{Two examples of non-convex domains.}
\label{fig:nonconvex}
\end{figure}
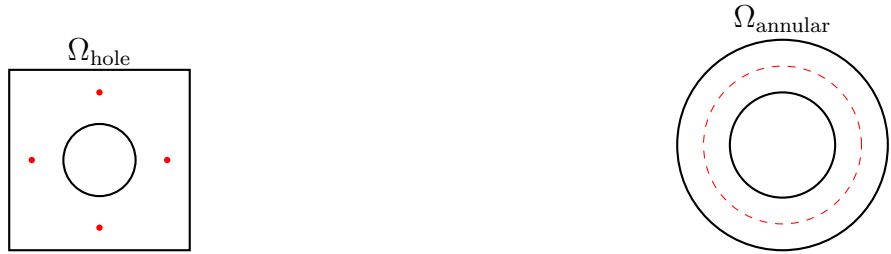

The following result abstracts the essential structure of the uniqueness proof given in the paper. It shows that the uniqueness of the critical point follows solely from the concavity of the distance function, the convexity and monotonicity properties of the integrand.

\begin{theorem} \label{thm:main-2}
Let $\Omega\subset\mathbb{R}^n$ be a bounded convex domain (no further regularity is required). For $\xi\in\Omega$ and a direction $\omega\in\mathbb{S}^{n-1}$, denote by $\rho(\xi,\omega)$ the distance from $\xi$ to the boundary $\partial\Omega$ in the direction $\omega$,
\begin{equation*}
\rho(\xi,\omega)=\sup\{t>0 \mid \xi+s\omega\in\Omega\ \text{for all }0\le s<t\}.
\end{equation*}
Let $f:(0,\infty)\to\mathbb{R}$ be a continuously differentiable function which is convex and decreasing. Define
\begin{equation*}
\Phi(\xi)=\int_{\mathbb{S}^{n-1}} f\left(\rho(\xi,\omega)\right)\mathrm{d}\omega,\quad \xi\in\Omega.
\end{equation*}
Then the following hold:
\begin{enumerate}
\item $\Phi$ is convex on $\Omega$.
\item If in addition $f$ is strictly convex, then $\Phi$ is strictly convex on $\Omega$.
\item If $\Phi(\xi)\to+\infty$ as $\xi\to\partial\Omega$, then $\Phi$ attains a unique global minimum in $\Omega$. Moreover, if $\Phi$ is differentiable, this minimizer is the only critical point of $\Phi$ in $\Omega$.
\end{enumerate}
\end{theorem}

In the previous, we analyzed the critical points of $\psi_\Omega$ for a bounded convex domain. A natural generalization is to consider domains consisting of several concentric annuli. Such domains are still rotationally symmetric, but their topology is richer. The following result illustrates how the topological complexity directly translates into a lower bound for the number of critical points.
 
\begin{theorem}\label{the:multiring}
Let $m\in\mathbb{N}_{+}$ and let $0<a_1<b_1<a_2<b_2<\cdots<a_m<b_m$. Consider the union of concentric annuli
\begin{equation*}
\Omega = \bigcup_{i=1}^m \{ x\in\mathbb{R}^n \mid a_i<|x|<b_i \}\subset\mathbb{R}^n,\quad n\ge 2.
\end{equation*}
Then the function $\psi_\Omega$ defined by
\begin{equation*}
\psi_\Omega(\xi)=\int_{\mathbb{R}^n\setminus\Omega}\frac{\mathrm{d}x}{|x-\xi|^{2n}},\quad \xi\in\Omega,
\end{equation*}
has exactly $m$ critical spheres $\{|\xi|=r_i\}$ with $r_i\in(a_i,b_i)$, $i=1,\cdots,m$. In particular, the number of critical points (each sphere counted as one critical set) equals the number of annular components.
\end{theorem}

The paper is organized as follows. In Section \ref{section-2}, we present some illustrative examples that highlight the role of symmetry and convexity. In Section \ref{section-3}, we introduce the distance function $\rho(\xi,\omega)$ and establish its key properties. 
Section \ref{section-4} is devoted to the spherical reduction, where we prove the strict convexity of the reduced functional and deduce the uniqueness of the critical point, thereby completing the proof of Theorem \ref{thm:main}. 
Finally, Appendix \ref{section-appendix} contains auxiliary lemmas on the distance function, which are of independent interest and also provide an alternative proof of the main result. Appendix \ref{section-com} details the computation of an integral appearing in the main text.

\subsection*{Notation} For vectors $u,v\in\mathbb{R}^n$, we denote by $u\otimes v$ the matrix with entries $(u\otimes v)_{ij}=u_i v_j$. The symbol $^\top$ stands for the transpose of a vector or a matrix.

\section{Illustrative examples and the role of convexity}
\label{section-2}

In this section, we examine several concrete domains to illustrate the meaning and the power of our uniqueness theorem, Theorem \ref{thm:main}.
These examples serve to highlight the crucial role of symmetry in locating the critical point, as well as the necessity of the abstract approach for establishing uniqueness.
In particular, we consider the annulus and multiple annuli, obtaining precise information about their critical points and thereby proving Theorem \ref{the:multiring}.

Let $\Omega = B_R(0)= \{ x\in\mathbb{R}^n \mid |x|<R \}$ be the open ball of radius $R>0$. 
Because of rotational symmetry, $\psi_\Omega(\xi)$ depends only on $|\xi|=r$. 
Writing $\xi = r e_1$ with $e_1=(1,0,\cdots,0)$ and $0\le r<R$, we set
\begin{equation*}
\psi(r)=\psi_\Omega(r e_1)=\int_{|x|\ge R}\frac{\mathrm{d}x}{|x-re_1|^{2n}}.
\end{equation*}

\begin{proposition}\label{prop:ball}
For the ball $\Omega = B_R(0)$, one has $\psi'(r)>0$ for all $r\in(0,R)$ and $\psi'(0)=0$. 
Consequently the only critical point of $\psi_\Omega$ is the origin $\xi=0$.
\end{proposition}

\begin{proof}
Differentiating under the integral sign gives (see also Lemma \ref{lem:psi_differentiable})
\begin{equation}
\psi'(r)=2n\int_{|x|\ge R}\frac{x_1-r}{|x-re_1|^{2n+2}}\mathrm{d}x= -\int_{|x|\ge R}\frac{\partial}{\partial x_1}|x-re_1|^{-2n}\mathrm{d}x. \label{eq:ball_deriv}
\end{equation}

For a fixed $x'=(x_2,\cdots,x_n)\in\mathbb{R}^{n-1}$ set
\begin{equation*}
a=
\begin{cases}
\sqrt{R^2-|x'|^2} &\text{if }|x'|\le R, \\
0 &\text{if }|x'| > R.
\end{cases} 
\end{equation*} 
If $|x'|>R$, the condition $|x|\ge R$ holds for every $x_1\in\mathbb{R}$. Hence the inner integral over $x_1$ becomes
\begin{equation*}
\int_{-\infty}^{\infty}\frac{\partial}{\partial x_1}|(x_1,x')-re_1|^{-2n}\mathrm{d}x_1
= \lim_{L\to\infty}\left(|(L,x')-re_1|^{-2n}-|(-L,x')-re_1|^{-2n}\right)=0.
\end{equation*}
Therefore the contribution of such $x'$ to the whole integral is zero. 
If $|x'|\le R$,  the domain $|x|\ge R$ splits into $x_1\ge a$ and $x_1\le -a$. Then
\begin{equation*}
-\int_{|x_1|\ge a}\frac{\partial}{\partial x_1}|x-re_1|^{-2n}\mathrm{d}x_1 \\
= \left(|(a,x')-re_1|^{-2n}-|(-a,x')-re_1|^{-2n}\right).
\end{equation*}
Hence
\begin{equation}
\psi'(r)= \int_{|x'|\le R}\left(|(a,x')-re_1|^{-2n}-|(-a,x')-re_1|^{-2n}\right)\mathrm{d}x'. \label{eq:ball_final}
\end{equation}

Notice that
\begin{equation*}
|(a,x')-re_1|^2 = (a-r)^2+|x'|^2, \quad
|(-a,x')-re_1|^2 = (a+r)^2+|x'|^2.
\end{equation*}
For $a\ge0$ and $r>0$ we have $a+r>|a-r|$, and therefore
\begin{equation*}
|(a,x')-re_1| < |(-a,x')-re_1|.
\end{equation*}
The function $t\mapsto t^{-2n}$ is strictly decreasing, so
\begin{equation*}
|(a,x')-re_1|^{-2n} > |(-a,x')-re_1|^{-2n}.
\end{equation*}
Thus the integrand in \eqref{eq:ball_final} is positive for every $x'$ with $|x'|\le R$. 
Consequently $\psi'(r)>0$ for all $r\in(0,R)$.

Finally, at $r=0$, formula \eqref{eq:ball_deriv} becomes
\begin{equation*}
\psi'(0)=2n\int_{|x|\ge R}\frac{x_1}{|x|^{2n+2}}\mathrm{d}x=0,
\end{equation*}
because the integrand is odd in $x_1$ and the domain is symmetric. 
Thus $\psi$ is strictly increasing on $[0,R)$ and the only critical point is $r=0$, that is, $\xi=0$. 
\end{proof}

\begin{remark}
The above computation is elementary but relies heavily on the radial symmetry of the ball. The function $t\mapsto t^{-2n}$ is strictly decreasing
It provides a direct verification of our main theorem in a highly symmetric situation.
\end{remark}

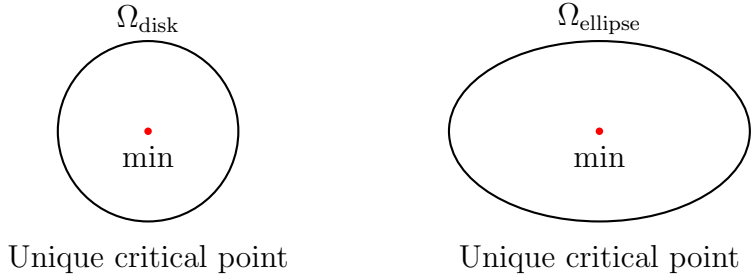
\begin{figure}[htbp]
\centering
\begin{tikzpicture}

\begin{scope}[xshift=-3cm]
\draw[thick] (0,0) circle (1.2);
\node[circle, fill=red, inner sep=1pt, label={below:min}] at (0,0) {};
\node at (0,1.5) {$\Omega_{\text{disk}}$};
\node at (0,-1.7) {Unique critical point};
\end{scope}

\begin{scope}[xshift=3cm]
\draw[thick] (0,0) ellipse (2 and 1.2);
\node[circle, fill=red, inner sep=1pt, label={below:min}] at (0,0) {};
\node at (0,1.5) {$\Omega_{\text{ellipse}}$};
\node at (0,-1.7) {Unique critical point};
\end{scope}
\end{tikzpicture}
\caption{The disk $\Omega_{\text{disk}}$ and the ellipse $\Omega_{\text{ellipse}}$ are convex domains.}
\label{fig:convex_examples}
\end{figure}

Consider now an ellipsoid centred at the origin,
\begin{equation*}
\Omega=\left\{x\in\mathbb{R}^n \mid \sum_{i=1}^n\frac{x_i^2}{a_i^2}<1\right\},\quad a_i>0, i=1, \cdots, n.
\end{equation*}
The domain is symmetric with respect to the origin: $x\in\Omega$ if and only if $-x\in\Omega$.

\begin{proposition}\label{lem:ellipsoid_symmetry}
For the ellipsoid defined above, one has $\psi_\Omega(-\xi)=\psi_\Omega(\xi)$ for every $\xi\in\Omega$. 
Hence the origin is a critical point.
\end{proposition}

\begin{proof}
Using the symmetry of the domain,
\begin{equation*}
\psi_\Omega(-\xi) = \int_{\mathbb{R}^n\setminus\Omega}\frac{\mathrm{d}x}{|x+\xi|^{2n}} 
= \int_{\mathbb{R}^n\setminus\Omega}\frac{\mathrm{d}y}{|y-\xi|^{2n}} = \psi_\Omega(\xi),
\end{equation*}
where we performed the change of variables $y=-x$. 
Differentiating the identity $\psi_\Omega(-\xi)=\psi_\Omega(\xi)$ with respect to $\xi$ and evaluating at $\xi=0$ gives $\nabla\psi_\Omega(0)=0$. 
Thus $\xi=0$ is a critical point.
\end{proof}

\begin{remark}
Note that symmetry alone suffices to determine a critical point, even when no closed‑form expression is available. On the other hand, Theorem \ref{thm:main} asserts that this point is necessarily unique, as illustrated in Figure \ref{fig:convex_examples}.
\end{remark}

Now, we analyze the annular domain $\Omega = \{ x \in \mathbb{R}^n \mid a < |x| < b \}$ with $0 < a < b$. By rotational symmetry, $\psi_\Omega(\xi)$ depends only on $r = |\xi|$, so we set $\psi(r) = \psi_\Omega(r e_1)$.

\begin{proposition}\label{prop:annulus}
Let $\Omega = \{ x\in\mathbb{R}^n \mid a<|x|<b \}$ with $0<a<b$. Then $\psi''(r)>0$ for every $r\in(a,b)$. Consequently $\psi$ is strictly convex, and $\psi_\Omega$ has exactly one critical sphere $\{|\xi|=r_0\}$ with $r_0 \in (a,b)$.
\end{proposition}

\begin{proof}
Recall that
\begin{equation*}
\psi(r)=\int_{|x|\le a}\frac{\mathrm{d}x}{|x-re_1|^{2n}}+\int_{|x|\ge b}\frac{\mathrm{d}x}{|x-re_1|^{2n}}\equiv I_{\text{in}}(r)+I_{\text{out}}(r). \label{eq:psi_split}
\end{equation*}
For $x=\rho\omega$ with $\rho=|x|$ and $\omega\in\mathbb{S}^{n-1}$, let $\theta$ be the angle between $\omega$ and $e_1$, that is, $\cos\theta=\langle\omega,e_1\rangle$.
The classical expansion (see \cite[6.4.9, 6.4.10]{AndrewsAskeyRoy1999}) gives
\begin{equation*}
\sum_{k=0}^{\infty} C_k^{(\lambda)}(t) s^k = (1 - 2ts + s^2)^{-\lambda}, \quad |s| <1,
\end{equation*}
where $C_k^{(\lambda)}$ is called Gegenbauer polynomials. Moreover, $C_k^{(n)}(\cos \theta) \leq C_k^{(n)}(1) = \frac{(2n)_k}{k!} \sim C_n k^{2n-1}$.
we set $\lambda = n$ and $t = \cos\theta$. For the case $\rho < r$, set $s = \rho/r$, we obtain 
\begin{equation*}
\frac{1}{|x - r e_1|^{2n}}
= \frac{1}{r^{2n}}\left(1 - 2s\cos\theta + s^2\right)^{-n}
= \sum_{k=0}^{\infty} C_k^{(n)}(\cos\theta) \frac{\rho^k}{r^{2n+k}}.
\end{equation*} 
When $\rho > r$, we interchange the roles of $\rho$ and $r$ and obtain
\begin{equation*}
\frac{1}{|x - r e_1|^{2n}}= \frac{1}{\rho^{2n}}\left(1 - 2\frac{r}{\rho}\cos\theta + \frac{r^2}{\rho^2}\right)^{-n}
= \sum_{k=0}^{\infty} C_k^{(n)}(\cos\theta) \frac{r^k}{\rho^{2n+k}}.
\end{equation*}
In both cases, we have
\begin{equation} \label{eq:expansion}
\frac{1}{|x - r e_1|^{2n}}= \sum_{k=0}^{\infty} C_k^{(n)}(\cos\theta)
\frac{\left(\min\{\rho,r\}\right)^k}{\left(\max\{\rho,r\}\right)^{2n+k}}.
\end{equation}
The series converges absolutely and uniformly on compact subsets where $\rho \neq r$. 
Define
\begin{equation}
A_k = \int_{\mathbb{S}^{n-1}} C_k^{(n)}(\cos\theta)\mathrm{d}\omega. \label{eq:Ak}
\end{equation}

\textit{Inner region $|x|\le a$}.
Here $\rho\le a<r$, so $\min\{\rho,r\}=\rho$, $\max\{\rho,r\}=r$. Integrating \eqref{eq:expansion} over the sphere and using \eqref{eq:Ak},
\begin{equation*}
\int_{\mathbb{S}^{n-1}}\frac{\mathrm{d}\omega}{|x-re_1|^{2n}} = \sum_{k=0}^{\infty} A_k\frac{\rho^k}{r^{2n+k}}. \label{eq:angular_in}
\end{equation*}
Substituting into the radial integral,
\begin{equation} \begin{aligned}
I_{\text{in}}(r) =& \int_{|x|\le a}\frac{\mathrm{d}x}{|x-re_1|^{2n}}
= \int_0^a \rho^{n-1}\left(\int_{\mathbb{S}^{n-1}}\frac{\mathrm{d}\omega}{|x-re_1|^{2n}}\right)\mathrm{d}\rho \\
=& \sum_{k=0}^{\infty} A_k\frac{1}{r^{2n+k}}\int_0^a \rho^{n-1+k}\mathrm{d}\rho
= \sum_{k=0}^{\infty} \frac{A_k}{n+k}\frac{a^{n+k}}{r^{2n+k}}. \label{eq:I_in}
\end{aligned} \end{equation}

\textit{Outer region $|x|\ge b$}.
Here $\rho\ge b>r$, so $\min\{\rho,r\}=r$, $\max\{\rho,r\}=\rho$. Hence
\begin{equation*}
\int_{\mathbb{S}^{n-1}}\frac{\mathrm{d}\omega}{|x-re_1|^{2n}} = \sum_{k=0}^{\infty} A_k\frac{r^k}{\rho^{2n+k}}. \label{eq:angular_out}
\end{equation*}
Then
\begin{equation} \begin{aligned}
I_{\text{out}}(r) =& \int_{|x|\ge b}\frac{\mathrm{d}x}{|x-re_1|^{2n}}
= \int_b^\infty \rho^{n-1}\left(\int_{\mathbb{S}^{n-1}}\frac{\mathrm{d}\omega}{|x-re_1|^{2n}}\right)\mathrm{d}\rho \\
=& \sum_{k=0}^{\infty} A_kr^k\int_b^\infty \rho^{-n-1-k}\mathrm{d}\rho
= \sum_{k=0}^{\infty} \frac{A_k}{n+k}\frac{r^k}{b^{n+k}}. \label{eq:I_out}
\end{aligned} \end{equation}

Adding \eqref{eq:I_in} and \eqref{eq:I_out} we obtain the explicit representation
\begin{equation*}
\psi(r)=\sum_{k=0}^{\infty}\frac{A_k}{n+k}\left( \frac{a^{n+k}}{r^{2n+k}} + \frac{r^k}{b^{n+k}} \right). \label{eq:psi_explicit}
\end{equation*}
Differentiating term by term gives
\begin{equation} \begin{aligned}
\psi'(r)=&\sum_{k=0}^{\infty}\frac{A_k}{n+k}\left( -(2n+k)\frac{a^{n+k}}{r^{2n+k+1}} + k\frac{r^{k-1}}{b^{n+k}} \right), \label{eq:psi_deriv}
\end{aligned} \end{equation}
and
\begin{equation} \begin{aligned}
\psi''(r)=&\sum_{k=0}^{\infty}\frac{A_k}{n+k}\left( (2n+k)(2n+k+1)\frac{a^{n+k}}{r^{2n+k+2}} + k(k-1)\frac{r^{k-2}}{b^{n+k}} \right). \label{eq:psi_second}
\end{aligned} \end{equation}
All terms in \eqref{eq:psi_second} are non‑negative and at least one positive, see Appendix \ref{section-com}. Hence $\psi''(r)>0$ on $(a,b)$ and therefore $\psi'$ is strictly increasing. Therefore, the function $\psi'$ is continuous, strictly increasing. 

Moreover, using \eqref{eq:psi_deriv}, we find that
\begin{equation*}
\lim_{r\to a^+}\psi'(r)=-\infty, \quad
\lim_{r\to b^-}\psi'(r)=+\infty. 
\end{equation*}
By the intermediate value theorem there exists a unique $r_0\in(a,b)$ such that $\psi'(r_0)=0$. This $r_0$ is the unique critical point of $\psi$. By rotational symmetry, it corresponds to a unique critical sphere $\{|\xi|=r_0\}$ for $\psi_\Omega$. Thus the proposition is proved.
\end{proof}

\begin{figure}[htbp]
\centering
\begin{tikzpicture}

\begin{scope}[xshift=-3cm]
\fill[white] (0,0) circle (1.5) (0,0) circle (0.8);
\draw[thick] (0,0) circle (1.5);
\draw[thick] (0,0) circle (0.8);

\draw[dashed, red] (0,0) circle (1.15);
\node at (0,1.8) {$\Omega_{\text{single}}$};
\node at (0,-2.0) {Single annulus};
\end{scope}

\begin{scope}[xshift=3cm, scale=0.35]

\fill[white] (0,0) circle (5) (0,0) circle (4);
\draw[thick] (0,0) circle (5);
\draw[thick] (0,0) circle (4);

\fill[white] (0,0) circle (3) (0,0) circle (2);
\draw[thick] (0,0) circle (3);
\draw[thick] (0,0) circle (2);

\fill[white] (0,0) circle (1.5) (0,0) circle (0.5);
\draw[thick] (0,0) circle (1.5);
\draw[thick] (0,0) circle (0.5);

\draw[dashed, red, thick] (0,0) circle (4.5);
\draw[dashed, red, thick] (0,0) circle (2.5);
\draw[dashed, red, thick] (0,0) circle (1.0);

\node at (0,5.8) {$\Omega_{\text{multiple}}$};
\node at (0,-5.8) {Multiple annuli};
\end{scope}
\end{tikzpicture}
\caption{Left: a single annular domain with a critical sphere. Right: multiple concentric annular domains, each with a critical sphere.}
\label{fig:annulus_single_vs_multi}
\end{figure}
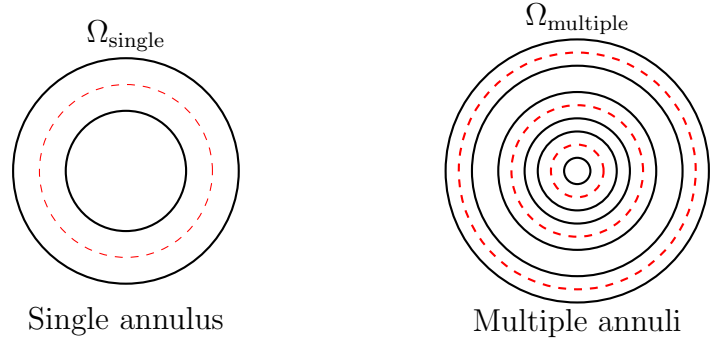

\begin{proof}[Proof of Theorem \ref{the:multiring}]
By rotational symmetry, $\psi_\Omega(\xi)$ depends only on $r=|\xi|$, we write $\psi(r)=\psi_\Omega(re_1)$. The domain of $\psi$ is the disjoint union $\bigcup_{i=1}^m (a_i,b_i)$. 

For a single annulus it was shown (see Proposition \ref{prop:annulus}) that $\psi''(r)>0$ on $(a_i,b_i)$ for every choice of $a_i,b_i$. Consequently $\psi'$ is strictly increasing and can have at most one zero. The limits $\psi'(a_i^+)=-\infty$ and $\psi'(b_i^-)=+\infty$ guarantee the existence of a unique zero $r_i\in(a_i,b_i)$. Hence each interval contributes exactly one critical sphere. 

Thus the total number of critical spheres of $\psi_\Omega$ is $m$. By rotational symmetry, every point on a critical sphere is a critical point. Consequently, $\psi_\Omega$ admits infinitely many critical points but only finitely many critical spheres (see Figure \ref{fig:annulus_single_vs_multi}).
\end{proof}

\section{Notation and preliminaries} \label{section-3}

Let $\Omega \subset \mathbb{R}^n$ ($n\ge 2$) be a bounded convex open set. For $\xi \in \Omega$ and a direction $\omega \in \mathbb{S}^{n-1}$, define $\rho(\xi,\omega)$ as the distance from $\xi$ to $\partial\Omega$ in the direction $\omega$, that is,
\begin{equation*} \label{eq:defrho1}
\rho(\xi,\omega) = \sup\{ t > 0 \mid \xi + t\omega \in \Omega\}.
\end{equation*}
Since $\Omega$ is bounded and open, $\rho(\xi,\omega)$ is well-defined and finite for all $\xi \in \Omega$ and every $\omega \in \mathbb{S}^{n-1}$. Moreover, the convexity of $\Omega$ implies that the ray $\xi + t\omega$ ($t > 0$) meets $\partial\Omega$ exactly once. Consequently, one also has the equivalent characterization
\begin{equation*} \label{eq:defrho}
\rho(\xi,\omega) = \sup\{ t > 0 \mid \xi + s\omega \in \Omega \text{ for all } 0 \le s < t \}.
\end{equation*}
For an intuitive understanding of $\rho$, see Figure \ref{fig:rho_definition}.
For further details on the basic properties, we refer the reader to \cite[Section 1.7]{Schneider2014} and the references therein, as well as to \cite{NiculescuPersson2018, Rockafellar1970}.

\begin{figure}[htbp]
\centering
\begin{tikzpicture}[
 scale=1.2,
 dot/.style={circle, fill=black, inner sep=1pt},
 boundary/.style={thick}
]
 
\draw[boundary] (0,0) ellipse (2cm and 1.2cm);
\node at (2.0,1) {$\Omega$};

\node[dot, label=left:{$\xi$}] (xi) at (0.3,0) {};

\draw[->, >=Stealth, thick, blue] (xi) -- ++(-30:1.5) coordinate (y);
\node[dot, label={[right]:$y$}] at (y) {};
\draw[dashed, blue] (xi) -- (y);
\node[blue] at (1.2,0) {$\rho(\xi,\omega)$};

\node at (-1.8,0.9) {$\partial\Omega$};
\end{tikzpicture}
\caption{Convex domain $\Omega$ with interior point $\xi$. }
\label{fig:rho_definition}
\end{figure}
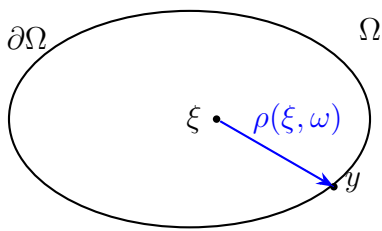

We begin by recalling some basic properties of the distance function $\rho(\xi,\omega)$ defined for a bounded convex domain $\Omega\subset\mathbb{R}^n$.

\begin{lemma} \label{lem:positive}
The function $\rho:\Omega\times\mathbb{S}^{n-1}\to(0,\infty)$ is positive. More precisely, for every $\xi\in\Omega$ there exists $\delta_\xi>0$ such that $\rho(\xi,\omega)\ge\delta_\xi$ for all $\omega\in\mathbb{S}^{n-1}$. Consequently, $\rho$ is bounded away from zero on any compact subsets of $\Omega$.
\end{lemma}

\begin{proof}
Fix $\xi\in\Omega$. Since $\Omega$ is open, there exists $r>0$ such that the open ball $B_r(\xi)$ is contained in $\Omega$. For any direction $\omega\in\mathbb{S}^{n-1}$, the segment $\{\xi+t\omega \mid 0\le t<r\}$ lies entirely in $B_r(\xi)$, hence in $\Omega$. By definition of $\rho(\xi,\omega)$, we must have $\rho(\xi,\omega)\ge r>0$. Thus $\rho(\xi,\omega)\ge r$ for every $\omega \in \mathbb{S}^{n-1}$, establishing positivity.
\end{proof}

\begin{lemma} \label{lem:conw}
For any fixed $\xi\in\Omega$, the function $\omega\mapsto\rho(\xi,\omega)$ is continuous on $\mathbb{S}^{n-1}$.
\end{lemma}

\begin{proof}
Fix $\xi\in\Omega$ and let $\omega_k\to\omega$ in $\mathbb{S}^{n-1}$. Set $t_k:=\rho(\xi,\omega_k)$. Since $\Omega$ is bounded, $\{t_k\}$ is bounded and hence has a convergent subsequence. Passing to a subsequence, we may assume that $t_k\to t$ for some $t$. Set $y_k:=\xi+t_k\omega_k\in\partial\Omega$. Passing to the limit and using the closedness of $\partial\Omega$, we obtain 
\begin{equation*}
y:=\xi+t\omega\in\partial\Omega.
\end{equation*}

By definition of $\rho(\xi,\omega)$, we have $t\le\rho(\xi,\omega)$ because $y$ lies on the ray from $\xi$ in direction $\omega$ and belongs to $\partial\Omega$. Suppose for contradiction that $t<\rho(\xi,\omega)$. Choose $s$ such that $t<s<\rho(\xi,\omega)$. Then 
\begin{equation*}
\xi+s\omega\in\Omega.
\end{equation*} 
Because $\Omega$ is open, there exists a neighborhood $U$ of $\xi+s\omega$ contained in $\Omega$. The convergence $\omega_k\to\omega$ implies that for sufficiently large $k$, the points $\xi+s\omega_k$ are close to $\xi+s\omega$ and therefore lie in $U\subset\Omega$. Hence $\xi+s\omega_k\in\Omega$ for all large $k$, which contradicts the maximality of $t_k=\rho(\xi,\omega_k)$ because $s>t$ and $t_k\to t$. Thus $t=\rho(\xi,\omega)$.

We have shown that every convergent subsequence of $\{t_k\}$ converges to $\rho(\xi,\omega)$. Since $\{t_k\}$ is bounded, the whole sequence converges to $\rho(\xi,\omega)$. Hence $\omega\mapsto\rho(\xi,\omega)$ is continuous.
\end{proof}

\begin{lemma} \label{lem:concave}
For each fixed direction $\omega\in\mathbb{S}^{n-1}$, the map $\xi\mapsto\rho(\xi,\omega)$ is concave on $\Omega$.
\end{lemma}

\begin{proof}
Let $\xi_1,\xi_2\in\Omega$ and $\lambda\in[0,1]$. For any $\varepsilon>0$, select $t_1<\rho(\xi_1,\omega)$ and $t_2<\rho(\xi_2,\omega)$ such that
\begin{equation*}
t_1 > \rho(\xi_1,\omega)-\varepsilon,\quad t_2 > \rho(\xi_2,\omega)-\varepsilon.
\end{equation*}
By definition of $\rho(\xi,\omega)$, the points $\xi_1+t_1\omega$ and $\xi_2+t_2\omega$ lie in $\Omega$. The convexity of $\Omega$ implies that their convex combination
\begin{equation*}
\lambda(\xi_1+t_1\omega)+(1-\lambda)(\xi_2+t_2\omega)=\left(\lambda\xi_1+(1-\lambda)\xi_2\right)+\left(\lambda t_1+(1-\lambda)t_2\right)\omega
\end{equation*}
also belongs to $\Omega$. Hence, the ray starting at $\lambda\xi_1+(1-\lambda)\xi_2$ in direction $\omega$ remains inside $\Omega$ at least up to the parameter $\lambda t_1+(1-\lambda)t_2$. Therefore,
\begin{equation*}
\rho\left(\lambda\xi_1+(1-\lambda)\xi_2,\omega\right)\ge \lambda t_1+(1-\lambda)t_2 > \lambda\rho(\xi_1,\omega)+(1-\lambda)\rho(\xi_2,\omega)-\varepsilon.
\end{equation*}
Since $\varepsilon>0$ is arbitrary, we obtain
\begin{equation*}
\rho\left(\lambda\xi_1+(1-\lambda)\xi_2,\omega\right)\ge \lambda\rho(\xi_1,\omega)+(1-\lambda)\rho(\xi_2,\omega),
\end{equation*}
which establishes the concavity inequality.
\end{proof}

From Lemma \ref{lem:concave}, for each fixed $\omega \in \mathbb{S}^{n-1}$, the function $\xi \mapsto \rho(\xi, \omega)$ is concave on $\Omega$. It follows by \cite[Theorem 6.7]{Evans2025} that $\rho(\cdot, \omega)$ is locally Lipschitz continuous. With Lemma \ref{lem:conw} in hand, we have the following conclusion.

\begin{lemma} \label{lem:rho_continuous}
The function $\rho:\Omega\times\mathbb{S}^{n-1}\to(0,\infty)$ is continuous.
\end{lemma}

Now, using spherical coordinates centered at $\xi$, we write
\begin{equation*}
\psi_\Omega(\xi) = \int_{\mathbb{R}^n\setminus\Omega} \frac{\mathrm{d}x}{|x-\xi|^{2n}}
= \int_{\mathbb{S}^{n-1}} \int_{\rho(\xi,\omega)}^{+\infty} \frac{r^{n-1}}{r^{2n}} \mathrm{d}r \mathrm{d}\omega
= \int_{\mathbb{S}^{n-1}} \int_{\rho}^{+\infty} r^{-n-1} \mathrm{d}r \mathrm{d}\omega.
\end{equation*}
Evaluating the inner integral gives $\int_{\rho}^{+\infty} r^{-n-1} \mathrm{d}r = \frac{1}{n} \rho^{-n}$. Thus
\begin{equation*}
\psi_\Omega(\xi) = \frac{1}{n} \int_{\mathbb{S}^{n-1}} \rho(\xi,\omega)^{-n} \mathrm{d}\omega.
\end{equation*}
For convenience we define
\begin{equation*}
F(\xi) := n \psi_\Omega(\xi) = \int_{\mathbb{S}^{n-1}} \rho(\xi,\omega)^{-n} \mathrm{d}\omega.
\end{equation*}
Clearly, critical points of $\psi_\Omega$ coincide with those of $F$. 

\begin{remark}
Direct computation gives that
\begin{equation*}
\nabla\psi_\Omega(\xi)=2n\int_{\mathbb{R}^n\setminus\Omega}\frac{x-\xi}{|x-\xi|^{2n+2}} \mathrm{d}x,
\end{equation*}
and
\begin{equation*}
\nabla^2\psi_\Omega(\xi)=2n\int_{\mathbb{R}^n\setminus\Omega}\left(-\frac{I}{|x-\xi|^{2n+2}}+(2n+2)\frac{(x-\xi)(x-\xi)^{\top}}{|x-\xi|^{2n+4}}\right)\mathrm{d}x.
\end{equation*}
For any non-zero vector $v\in\mathbb{R}^n$, the associated quadratic form is
\begin{equation*}
v^{\top}\nabla^2\psi_\Omega(\xi)v
=2n\int_{\mathbb{R}^n\setminus\Omega}\left(-\frac{|v|^2}{|x-\xi|^{2n+2}}+(2n+2)\frac{(v\cdot(x-\xi))^2}{|x-\xi|^{2n+4}}\right)\mathrm{d}x.
\end{equation*}
The positivity of the whole integral is \textbf{not obvious} and cannot be decided by a pointwise estimate. However, in the Cartesian representation the convexity is hidden in the integration region $\mathbb{R}^n\setminus\Omega$ and does not appear explicitly in the algebraic structure of the integrand. The spherical reduction is an essential step.
\end{remark}

\section{Strict convexity and uniqueness}  \label{section-4}
In this section, we prove that the reduced functional $F(\xi)=n\psi_\Omega(\xi)$ is strictly convex on $\Omega$. The proof is divided into two steps. First, we establish convexity by exploiting the concavity of $\rho(\cdot,\omega)$ and the convexity of the power function $t^{-n}$. Then we refine the argument to obtain strict convexity. Finally, we show that strict convexity forces uniqueness of the critical point, thereby completing the proof of Theorem \ref{thm:main}.

\subsection{Convexity of $F$}

We now show that $F$ is strictly convex in $\Omega$. First, we establish convexity.
For simplicity, we set $g(t)=t^{-n}$ for $t>0$. Then $g'(t)=-nt^{-n-1}<0$ and $g''(t)=n(n+1)t^{-n-2}>0$, so $g$ is strictly convex and decreasing.

\begin{lemma} \label{lem:convexity}
For each fixed $\omega \in \mathbb{S}^{n-1}$, the function $\xi \mapsto \rho(\xi,\omega)^{-n}$ is convex on $\Omega$. Consequently, $F$ is convex on $\Omega$.
\end{lemma}

\begin{proof}
Fix $\omega \in \mathbb{S}^{n-1}$. The function $\xi \mapsto \rho(\xi,\omega)$ is concave and positive on $\Omega$ by Lemma \ref{lem:rho_continuous}. 
For any $\xi_1,\xi_2\in\Omega$ and $\lambda\in[0,1]$, concavity of $\rho(\cdot,\omega)$ yields
\begin{equation*}
\rho(\lambda\xi_1+(1-\lambda)\xi_2,\omega) \ge \lambda\rho(\xi_1,\omega)+(1-\lambda)\rho(\xi_2,\omega).
\end{equation*}
Since $g$ is decreasing,
\begin{equation} \label{eq:g}
g\left(\rho(\lambda\xi_1+(1-\lambda)\xi_2,\omega)\right) \le g\left(\lambda\rho(\xi_1,\omega)+(1-\lambda)\rho(\xi_2,\omega)\right).
\end{equation}
Now convexity of $g$ gives
\begin{equation*}
g\left(\lambda\rho(\xi_1,\omega)+(1-\lambda)\rho(\xi_2,\omega)\right) \le \lambda g(\rho(\xi_1,\omega)) + (1-\lambda)g(\rho(\xi_2,\omega)).
\end{equation*}
Combining the inequalities, we obtain
\begin{equation*}
\rho(\lambda\xi_1+(1-\lambda)\xi_2,\omega)^{-n} \le \lambda\rho(\xi_1,\omega)^{-n} + (1-\lambda)\rho(\xi_2,\omega)^{-n},
\end{equation*}
which proves that $\xi\mapsto\rho(\xi,\omega)^{-n}$ is convex. Integrating over $\mathbb{S}^{n-1}$, we conclude that $F$ is convex on $\Omega$.
\end{proof}

The following lemma is geometrically intuitive: moving the starting point along a fixed direction simply shifts the distance to the boundary by the same amount. However, for the sake of completeness, we provide a rigorous proof based solely on definition of $\rho(\xi,\omega)$.

\begin{lemma}\label{lem:rho_linear}
Let $\Omega\subset\mathbb{R}^n$ be a convex open set and let $\xi_1,\xi_2\in\Omega$ be two distinct points. Set
\begin{equation*}
\omega := \frac{\xi_2-\xi_1}{|\xi_2-\xi_1|}\in\mathbb{S}^{n-1}.
\end{equation*}
Then
\begin{equation*}
\rho(\xi_2,\omega) = \rho(\xi_1,\omega) - |\xi_1-\xi_2|.
\end{equation*}
\end{lemma}

\begin{proof}
Since $\Omega$ is convex, $\xi_2 = \xi_1 + t_0\omega$ with $t_0 := |\xi_1-\xi_2|$ belongs to $\Omega$. Hence $t_0 < \rho(\xi_1,\omega)$ by definition of $\rho(\xi_1,\omega)$. 
For any $t\ge 0$, we have $\xi_2 + t\omega = \xi_1 + (t_0 + t)\omega$. By definition of $\rho(\xi_1,\omega)$,
\begin{equation*}
\xi_1 + s\omega \in \Omega \quad \text{for all } s < \rho(\xi_1,\omega).
\end{equation*}

If $t > \rho(\xi_1,\omega) - t_0$, then $t_0 + t > \rho(\xi_1,\omega)$ and consequently $\xi_2 + t\omega \notin \Omega$. Therefore every admissible $t$ for $\xi_2$ must satisfy $t \le \rho(\xi_1,\omega) - t_0$, which implies
\begin{equation}\label{eq:upper}
\rho(\xi_2,\omega) \le \rho(\xi_1,\omega) - t_0.
\end{equation}
If $t$ satisfies $0 \le t < \rho(\xi_1,\omega) - t_0$. Then $t_0 + t < \rho(\xi_1,\omega)$, so we have 
\begin{equation*}
\xi_2 + t\omega = \xi_1 + (t_0 + t)\omega \in \Omega.
\end{equation*} 
Hence every such $t$ is admissible for $\xi_2$, and therefore
\begin{equation}\label{eq:lower}
\rho(\xi_2,\omega) \ge \rho(\xi_1,\omega) - t_0.
\end{equation}
Combining \eqref{eq:upper} and \eqref{eq:lower} yields the desired equality.
\end{proof}

\begin{lemma} \label{lem:strictconvex}
$F$ is strictly convex on $\Omega$.
\end{lemma}

\begin{proof}
Take any distinct $\xi_1,\xi_2\in\Omega$ and $\lambda\in(0,1)$. Set 
\begin{equation*}
\omega_0 = \frac{\xi_2-\xi_1}{|\xi_2-\xi_1|},
\end{equation*} 
the unit direction from $\xi_1$ to $\xi_2$. Since $\Omega$ is convex, Lemma \ref{lem:rho_linear} implies that
\begin{equation*}
\rho(\xi_2,\omega_0) = \rho(\xi_1,\omega_0) - |\xi_1-\xi_2|.
\end{equation*}
Since $|\xi_1-\xi_2|>0$, we obtain 
$
\rho(\xi_1,\omega_0) \neq \rho(\xi_2,\omega_0).
$
By continuity of $\rho(\cdot,\cdot)$ (Lemma \ref{lem:rho_continuous}), there exists an open neighborhood $U$ of $\omega_0$ in $\mathbb{S}^{n-1}$ such that for every $\omega\in U$,
\begin{equation*}
\rho(\xi_1,\omega) \neq \rho(\xi_2,\omega).
\end{equation*}

Fix $\omega\in U$. Recall \eqref{eq:g} that 
\begin{equation*}
g\left(\rho(\lambda\xi_1+(1-\lambda)\xi_2,\omega)\right) \le g\left(\lambda\rho(\xi_1,\omega)+(1-\lambda)\rho(\xi_2,\omega)\right).
\end{equation*}
Since $g$ is strictly convex and $\rho(\xi_1,\omega)\neq\rho(\xi_2,\omega)$, we obtain
\begin{equation*}
g\left(\lambda\rho(\xi_1,\omega)+(1-\lambda)\rho(\xi_2,\omega)\right) < \lambda g(\rho(\xi_1,\omega)) + (1-\lambda)g(\rho(\xi_2,\omega)).
\end{equation*}
Combining the last two inequalities yields
\begin{equation*}
\rho(\lambda\xi_1+(1-\lambda)\xi_2,\omega)^{-n} < \lambda\rho(\xi_1,\omega)^{-n} + (1-\lambda)\rho(\xi_2,\omega)^{-n}, \quad \text{for every $\omega\in U$. }
\end{equation*}
For $\omega\notin U$, we only have the non-strict inequality from Lemma \ref{lem:convexity}. Integrating over $\mathbb{S}^{n-1}$ and noting that $U$ has positive measure, we get
\begin{equation*}
F(\lambda\xi_1+(1-\lambda)\xi_2) < \lambda F(\xi_1) + (1-\lambda)F(\xi_2).
\end{equation*}
Thus $F$ is strictly convex on $\Omega$.
\end{proof}

\begin{remark}
We now present a \textbf{formal} computation. Assuming that $\rho$ is sufficiently smooth and that the interchange of differentiation and integration is justified (this can be made rigorous for $C^2$ convex domains, see Appendix \ref{section-appendix}), we obtain
\begin{equation*}
\nabla_\xi\left(\rho(\xi,\omega)^{-n}\right) = -n \rho(\xi,\omega)^{-n-1} \nabla_\xi\rho.
\end{equation*}
Differentiating once more,
\begin{equation*}
\nabla_\xi^2\left(\rho(\xi,\omega)^{-n}\right)
= n(n+1)\rho(\xi,\omega)^{-n-2} \nabla_\xi\rho\otimes\nabla_\xi\rho
- n\rho(\xi,\omega)^{-n-1} \nabla_\xi^2\rho.
\end{equation*}
The first term is a positive semidefinite rank-1 matrix, and the second term is positive semidefinite because $\nabla_\xi^2\rho$ is negative semidefinite. Therefore each $\nabla_\xi^2\left(\rho(\xi,\omega)^{-n}\right)$ is positive semidefinite, and integrating over $\omega$ yields that $\nabla^2F(\xi)$ is positive semidefinite. 
\end{remark}

\subsection{Uniqueness of the critical point}

With the strict convexity of $F$ established in the previous subsection, we now prove the uniqueness of the critical point of $\psi_\Omega$. 

\begin{lemma}\label{lem:psi_differentiable}
The function $\psi_\Omega:\Omega\to\mathbb{R}$ defined by
\begin{equation*}
\psi_\Omega(\xi)=\int_{\mathbb{R}^n\setminus\Omega}\frac{\mathrm{d}x}{|x-\xi|^{2n}}
\end{equation*}
is continuously differentiable on $\Omega$. Moreover, for every $\xi\in\Omega$,
\begin{equation*}
\nabla\psi_\Omega(\xi)=2n\int_{\mathbb{R}^n\setminus\Omega}\frac{x-\xi}{|x-\xi|^{2n+2}} \mathrm{d}x.
\end{equation*}
\end{lemma}

\begin{proof}
Fix $\xi_0\in\Omega$ and choose a compact neighborhood $K\subset\Omega$ of $\xi_0$. Set $\delta:=\operatorname{dist}(K,\partial\Omega)>0$. For any $\xi\in K$ and $x\in\mathbb{R}^n\setminus\Omega$, we have $|x-\xi|\ge\delta$. The integrand $f(\xi,x):=|x-\xi|^{-2n}$ is smooth in $\xi$, and its partial derivatives satisfy
\begin{equation*}
\left|\partial_{\xi_i}f(\xi,x)\right|=\left|2n (x_i-\xi_i)|x-\xi|^{-2n-2}\right|\le 2n|x-\xi|^{-2n-1}.
\end{equation*}
Using the dominated convergence theorem, $\psi_\Omega$ is differentiable at every $\xi\in K$ and the gradient formula holds. The continuity of $\nabla\psi_\Omega$ follows from the continuity of the integrand in $\xi$ and the uniform bound.
\end{proof}

\begin{proof}[Proof of Theorem \ref{thm:main}]
We have shown that
\begin{equation*}
F(\xi)=n\psi_\Omega(\xi)=\int_{\mathbb{S}^{n-1}}\rho(\xi,\omega)^{-n} \mathrm{d}\omega
\end{equation*}
is strictly convex on $\Omega$ (Lemma \ref{lem:strictconvex}). Note that we have not established the everywhere differentiability of $F$ directly. However, since $F = n\psi$ and $\psi$ is differentiable, we obtain the differentiability of $F$. 
Recall that for a differentiable convex function $F$ on $\Omega$, we have the first-order inequality
\begin{equation*}
F(y) \ge F(x) + \nabla F(x)^\top (y-x) \quad \text{for all } x,y\in\Omega
\end{equation*}
(see \cite[Theorem 1.5.12]{Schneider2014}). If $x$ is a critical point, that is, $\nabla F(x)=0$, then $F(y)\ge F(x)$ for every $y\in\Omega$. Hence any critical point of a convex function is automatically a global minimizer.

Now, suppose for contradiction that $F$ possesses two distinct critical points $\xi_1,\xi_2\in\Omega$. By the previous observation both are global minimizers, so $F(\xi_1)=F(\xi_2)=m$, where $m$ denotes the global minimum value. Because $F$ is strictly convex, for any $t\in(0,1)$, we have
\begin{equation*}
F\left((1-t)\xi_1+t\xi_2\right) < (1-t)F(\xi_1)+tF(\xi_2)=m.
\end{equation*}
But $(1-t)\xi_1+t\xi_2$ belongs to $\Omega$, contradicting definition of $m$. Therefore two distinct critical points cannot exist. 
Consequently, $F$ has exactly one critical point, and that point is its global minimum. Since $\psi_\Omega = \frac{1}{n}F$, the same conclusion holds for $\psi_\Omega$. This completes the proof of Theorem \ref{thm:main}.
\end{proof}

\begin{proof}[The proof of Theorem \ref{thm:main-2}]
We split the proof into three parts.

\textit{1. Convexity of $\Phi$.} Fix $\omega\in\mathbb{S}^{n-1}$. Because $\Omega$ is convex, the function $\xi\mapsto\rho(\xi,\omega)$ is concave (Lemma \ref{lem:concave}). Since $f$ is convex and decreasing, the composition $g_\omega(\xi):=f(\rho(\xi,\omega))$ is convex. Indeed, for any $\xi_1,\xi_2\in\Omega$ and $\lambda\in[0,1]$,
\begin{equation*}
\rho(\lambda\xi_1+(1-\lambda)\xi_2,\omega)\ge \lambda\rho(\xi_1,\omega)+(1-\lambda)\rho(\xi_2,\omega)
\end{equation*}
by concavity. As $f$ is decreasing,
\begin{equation*} \begin{aligned}
g_\omega(\lambda\xi_1+(1-\lambda)\xi_2)= &f\left(\rho(\lambda\xi_1+(1-\lambda)\xi_2,\omega)\right) \\
\le &f\left(\lambda\rho(\xi_1,\omega)+(1-\lambda)\rho(\xi_2,\omega)\right).
\end{aligned} \end{equation*}
Using convexity of $f$,
\begin{equation*} \begin{aligned}
f\left(\lambda\rho(\xi_1,\omega)+(1-\lambda)\rho(\xi_2,\omega)\right)
\le &\lambda f(\rho(\xi_1,\omega))+(1-\lambda)f(\rho(\xi_2,\omega)) \\
=& \lambda g_\omega(\xi_1)+(1-\lambda)g_\omega(\xi_2).
\end{aligned} \end{equation*}
Thus $g_\omega$ is convex. Integrating over $\omega$ preserves convexity, so $\Phi$ is convex.

\textit{2. Strict convexity under extra conditions.} Take distinct $\xi_1,\xi_2\in\Omega$ and $\lambda\in(0,1)$. For every $\omega$ with $\rho(\xi_1,\omega)\neq\rho(\xi_2,\omega)$, the strict convexity of $f$ gives
\begin{equation*}
f\left(\lambda\rho(\xi_1,\omega)+(1-\lambda)\rho(\xi_2,\omega)\right)
< \lambda f(\rho(\xi_1,\omega))+(1-\lambda)f(\rho(\xi_2,\omega)).
\end{equation*}
Combining with the inequality from step 1, we obtain
\begin{equation*}
g_\omega(\lambda\xi_1+(1-\lambda)\xi_2)
< \lambda g_\omega(\xi_1)+(1-\lambda)g_\omega(\xi_2)
\end{equation*}
for those $\omega$. Also, for any two distinct points $\xi_1,\xi_2\in\Omega$ the set
\begin{equation*}
\{\omega\in\mathbb{S}^{n-1} \mid \rho(\xi_1,\omega)\neq\rho(\xi_2,\omega)\}
\end{equation*}
has positive Lebesgue measure (using Lemma \ref{lem:rho_continuous}).
Integrating over $\mathbb{S}^{n-1}$ yields
\begin{equation*}
\Phi(\lambda\xi_1+(1-\lambda)\xi_2)<\lambda\Phi(\xi_1)+(1-\lambda)\Phi(\xi_2),
\end{equation*}
proving strict convexity.

\textit{3. Uniqueness of the minimizer and of the critical point.} If $\Phi(\xi)\to\infty$ as $\xi \to \partial\Omega$, then $\Phi$ attains a minimum at some point $\xi_0\in\Omega$. Strict convexity implies that this minimizer is unique. Moreover, if $\Phi$ is differentiable at $\xi_0$, we have $\nabla\Phi(\xi_0)=0$. Any other critical point would also be a global minimizer by the first-order characterisation of convex functions, contradicting uniqueness. Hence $\Phi$ has exactly one critical point. 
\end{proof}

\appendix

\section{Further properties of $\rho$} \label{section-appendix}
This appendix gathers additional properties of the distance function $\rho(\xi,\omega)$. While not required for the proof of Theorem \ref{thm:main}, these results are of independent interest and, in fact, yield an alternative proof of the theorem for $C^2$ bounded convex domains. Specifically, the following lemmas reveal new connections between the differentiability of $\rho$ in $\xi$ and the regularity of the boundary. 

\begin{lemma} \label{lem:rho_lipschitz}
Let $\Omega\subset\mathbb{R}^n$ be a bounded convex domain. For each fixed $\omega\in\mathbb{S}^{n-1}$, the function $\xi\mapsto\rho(\xi,\omega)$ is locally Lipschitz continuous on $\Omega$. More precisely, for any compact set $K\subset\Omega$, set
\begin{equation*}
\delta:=\operatorname{dist}(K,\partial\Omega)>0.
\end{equation*}
Then for all $x,y\in K$ and every $\omega\in\mathbb{S}^{n-1}$,
\begin{equation*}
|\rho(x,\omega)-\rho(y,\omega)|\le L_K |x-y|,
\end{equation*}
where $L_K:=\frac{\operatorname{diam}(\Omega)}{\delta}$.
\end{lemma}

\begin{proof}
Fix $\omega\in\mathbb{S}^{n-1}$ and $x,y\in K$ with $x\neq y$. Let $\omega_0:=\frac{y-x}{|y-x|}$ be the unit direction from $x$ to $y$. Choose $\mu:=\frac{\delta}{|y-x|}+1>1$ and set
\begin{equation*}
z:=x+\mu(y-x)=x+\mu|y-x|\omega_0.
\end{equation*}
Then $|z-x|=\mu|y-x|=\delta+|y-x|$. Since $\operatorname{dist}(x,\partial\Omega)\ge\delta$, the point $z$ lies in $\Omega$. Moreover, 
$y=\frac{1}{\mu}z+(1-\frac{1}{\mu})x$.
By Lemma \ref{lem:concave}, we have
\begin{equation*}
\rho(y,\omega_0)\ge\frac{1}{\mu}\rho(z,\omega_0)+\left(1-\frac{1}{\mu}\right)\rho(x,\omega_0).
\end{equation*}
Since $\rho(z,\omega_0)\ge0$, we obtain
\begin{equation*}
\rho(x,\omega_0)-\rho(y,\omega_0)\le\frac{1}{\mu} \rho(x,\omega_0)\le\frac{1}{\mu}\sup_{\Omega}\rho(\cdot,\omega_0).
\end{equation*}
Because $\mu\ge\frac{\delta}{|y-x|}$, we have $\frac{1}{\mu}\le\frac{|y-x|}{\delta}$. Hence
\begin{equation*}
\rho(x,\omega_0)-\rho(y,\omega_0)\le\frac{|y-x|}{\delta}\sup_{\Omega}\rho(\cdot,\omega_0).
\end{equation*}
Interchanging the roles of $x$ and $y$ yields the same estimate for $\rho(y,\omega_0)-\rho(x,\omega_0)$. Therefore
\begin{equation*}
|\rho(x,\omega_0)-\rho(y,\omega_0)|\le\frac{\sup_{\Omega}\rho(\cdot,\omega_0)}{\delta} |x-y|.
\end{equation*}

Notice that $\rho$ is bounded on $\Omega\times\mathbb{S}^{n-1}$, because $\Omega$ is bounded and $\rho(\xi,\omega)\le\operatorname{diam}(\Omega)$ for all $\xi \in \Omega,\omega\in\mathbb{S}^{n-1}$. 
Then for all $x,y\in K$ and all $\omega\in\mathbb{S}^{n-1}$,
\begin{equation*}
|\rho(x,\omega)-\rho(y,\omega)|\le\frac{\operatorname{diam}(\Omega)}{\delta} |x-y|.
\end{equation*}
This completes the proof.
\end{proof}

Since $\rho(\cdot, \omega)$ is Lipschitz continuous on $K$ for each fixed $\omega$ by Lemma \ref{lem:rho_lipschitz}, it is differentiable almost everywhere in $K$ with respect to $\xi$ by Rademacher's theorem (see \cite[Theorem 3.2]{Evans2025}). The following lemma provides a uniform bound for this gradient, independent of $\omega$ and $\xi$, which implies that $\nabla_\xi\rho$ belongs to $L^\infty(K \times \mathbb{S}^{n-1})$. Let $\delta := \operatorname{dist}(K, \partial\Omega) > 0$ denote the distance from $K$ to the boundary of $\Omega$.

\begin{lemma} \label{lem:uniform_gradient_bound}
Let $\Omega \subset \mathbb{R}^n$ be a bounded convex open set. For any compact subset $K \subset \Omega$, there exists a constant $L_K > 0$, independent of $\omega$, such that for all $\omega \in \mathbb{S}^{n-1}$ and almost every $\xi \in K$, the gradient satisfies
\begin{equation*}
|\nabla_\xi \rho(\xi, \omega)| \leq L_K,
\end{equation*}
where $L_K:=\frac{\operatorname{diam}(\Omega)}{\delta}$.
Consequently, $\nabla_\xi \rho \in L^\infty(K \times \mathbb{S}^{n-1})$.
\end{lemma}

The following lemma shows that on any compact set $K\subset\Omega$, the quantity $\nu(y)\cdot\omega$ admits a uniform positive lower bound independent of $\omega$. This enables us to apply the implicit function theorem in Lemma \ref{lem:rho_C2}.

\begin{lemma}\label{lem:nu_omega_bound}
Let $\Omega\subset\mathbb{R}^n$ be a bounded convex domain with $C^1$ boundary. 
Then for every $\xi\in K$ and every $\omega\in\mathbb{S}^{n-1}$,
\begin{equation*}
\nu(y)\cdot\omega >0,
\end{equation*}
where $y=\xi+\rho(\xi,\omega)\omega\in\partial\Omega$, and $\nu(y)$ denotes the unit outward normal at $y$.

Moreover, for any compact set $K\subset\Omega$, there exists a constant $c_K>0$ such that
\begin{equation*}
\nu(y)\cdot\omega \ge c_K,
\end{equation*}
for every $\xi\in K$ and every $\omega\in\mathbb{S}^{n-1}$.
\end{lemma}

\begin{proof}
Fix $(\xi,\omega)\in\Omega\times\mathbb{S}^{n-1}$ and let $y$ be the corresponding boundary point. By convexity of $\Omega$, we have $\overline{\Omega}$ is convex by \cite[Theorem 1.1.10]{Schneider2014}. 
It follows by Supporting Hyperplane Theorem (see \cite[Theorem 1.3.2]{Schneider2014}), the outward normal $\nu(y)$ satisfies
\begin{equation*}
\nu(y)\cdot(z-y)\le 0\quad\text{for all }z\in\Omega,
\end{equation*}
with equality only if $z$ lies on the supporting hyperplane. Since $\xi$ is an interior point of $\Omega$, it does not belong to that hyperplane. Hence
\begin{equation*}
\nu(y)\cdot(\xi-y)<0.
\end{equation*}
But $\xi-y = -\rho(\xi,\omega)\omega$ and $\rho(\xi,\omega)>0$, therefore
\begin{equation*}
-\rho(\xi,\omega) \nu(y)\cdot\omega<0\quad\Longrightarrow\quad \nu(y)\cdot\omega>0.
\end{equation*}
Thus the function $f(\xi,\omega):=\nu(y)\cdot\omega$ is strictly positive on $\Omega\times\mathbb{S}^{n-1}$. 
The continuity of $f$ follows from the continuity of $\rho$ (Lemma \ref{lem:rho_continuous}), and the $C^1$ regularity of $\partial\Omega$, which guarantees that $\nu:\partial\Omega\to\mathbb{S}^{n-1}$ is continuous. 
Restricting $f$ to the compact set $K\times\mathbb{S}^{n-1}$, it attains a minimum $c_K$. Because $f$ is everywhere positive on this set, we must have $c_K>0$. An intuitive illustration is provided in Figure \ref{fig:nu_omega_correct}. This is the desired uniform lower bound.
\end{proof}

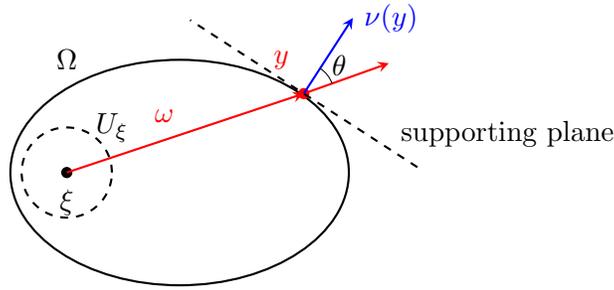
\begin{figure}[htbp]
\centering

\begin{tikzpicture}[scale=1.5, >=stealth, font=\small]
 
 \draw[thick] (0,0) ellipse (1.5 and 1.0);
 \node at (-1.0,1.0) {$\Omega$};
 \node at (0.9,1.0) {\textcolor{red}{$y$}};

 \coordinate (xi) at (-1.0,0.0);
 \node[circle, fill=black, inner sep=1.5pt, label=below:$\xi$] at (xi) {};

 \draw[dashed, thick] (xi) circle (0.4);
 \node at ($(xi)+(0.4,0.4)$) {$U_\xi$};

 \coordinate (y) at (1.1,0.7);
 \node[circle, fill=red, inner sep=1.5pt, label=above right:] at (y) {};

 \draw[->, thick, red] (xi) -- (y) node[midway, above left] {$\omega$};

 \coordinate (w_vec) at ($(y)-(xi)$); 

 \pgfmathsetmacro{\wx}{2.8} 
 \pgfmathsetmacro{\wy}{1.0}
 \pgfmathsetmacro{\wlen}{sqrt(\wx*\wx + \wy*\wy)}
 \pgfmathsetmacro{\ux}{\wx/\wlen}
 \pgfmathsetmacro{\uy}{\wy/\wlen}
 \pgfmathsetmacro{\extlen}{0.8}
 \coordinate (w_ext) at ($(y)+\extlen*(\ux,\uy)$);
 \draw[->, thick, red] (y) -- (w_ext);

 \pgfmathsetmacro{\nx}{1.8/6.25} 
 \pgfmathsetmacro{\ny}{1.0/2.25} 
 \pgfmathsetmacro{\nlen}{sqrt(\nx*\nx + \ny*\ny)}
 \pgfmathsetmacro{\nux}{\nx/\nlen}
 \pgfmathsetmacro{\nuy}{\ny/\nlen}

 \pgfmathsetmacro{\nulen}{0.8}
 \coordinate (nu) at (\nulen*\nux, \nulen*\nuy);
 \draw[->, thick, blue] (y) -- ++(nu) node[right] {$\nu(y)$};

 \pgfmathsetmacro{\tx}{-\nuy}
 \pgfmathsetmacro{\ty}{\nux}
 \coordinate (tan) at (\tx,\ty);

 \pgfmathsetmacro{\tanlen}{1.2}
 \draw[dashed, thick] ($(y)-\tanlen*(\tx,\ty)$) -- ($(y)+\tanlen*(\tx,\ty)$);
 \node at ($(y)+(-\tanlen*\tx + 0.8, -\tanlen*\ty + 0.3)$) {supporting plane};

 \coordinate (w_dir) at ($(y)+0.2*(\ux,\uy)$); 
 \coordinate (nu_dir) at ($(y)+0.2*(\nux,\nuy)$); 
 \pic[draw, "$\theta$", angle eccentricity=1.5, angle radius=0.4cm] {angle = w_dir--y--nu_dir};
\end{tikzpicture}
\caption{Geometry of $\nu(y)\cdot\omega>0$. The ray from $\xi$ in direction $\omega$ meets the boundary $\partial\Omega$ at $y$.}
\label{fig:nu_omega_correct}
\end{figure}

With this bound for $\nu(y)\cdot\omega$ established in Lemma \ref{lem:nu_omega_bound}, we can now prove that for each fixed direction $\omega$, the distance function $\xi\mapsto\rho(\xi,\omega)$ is actually $C^2$ on the whole domain $\Omega$, and its derivatives are uniformly bounded on compact sets. This goes far beyond the almost everywhere differentiability (Lemma \ref{lem:uniform_gradient_bound}) guaranteed by convexity and is essential for the subsequent analysis of the Hessian of $F$.

\begin{lemma} \label{lem:rho_C2}
Let $\Omega\subset\mathbb{R}^n$ be a bounded convex domain with $C^2$ boundary. For any fixed direction $\omega\in\mathbb{S}^{n-1}$, the function $\xi\mapsto\rho(\xi,\omega)$ is of class $C^2$ on $\Omega$. 
Moreover, for every $\xi\in\Omega$ and $\omega\in\mathbb{S}^{n-1}$, the first and second derivatives of $\rho$ are given by 
\begin{equation*}
\nabla_\xi\rho(\xi,\omega)=-\frac{\nu(y)}{\nu(y)\cdot\omega},
\end{equation*}
and
\begin{equation*}
\nabla_\xi^2\rho = -\frac{1}{\nu\cdot\omega}\nabla_\xi\nu + \frac{1}{(\nu\cdot\omega)^2} \nu\otimes\nabla_\xi(\nu\cdot\omega),
\end{equation*}
where $y=\xi+\rho(\xi,\omega)\omega$ and $\nu=\nu(y)$ denotes the outward unit normal at $y$.

Furthermore, for any compact set $K\subset\Omega$, there exists a constant $C_K>0$ such that
\begin{equation*}
|\nabla_\xi\rho(\xi,\omega)| \le C_K\quad\text{and}\quad |\nabla_\xi^2\rho(\xi,\omega)| \le C_K,
\end{equation*}
for any $\xi\in K$ and every $\omega\in\mathbb{S}^{n-1}$.
\end{lemma}

\begin{proof}
Fix $(\xi_0,\omega)\in\Omega\times\mathbb{S}^{n-1}$ and let $t_0=\rho(\xi_0,\omega)$ and $y_0=\xi_0+t_0\omega\in\partial\Omega$. Since $\partial\Omega$ is $C^2$, there exists a neighborhood $U\subset\mathbb{R}^n$ of $y_0$ and a $C^2$ function $\gamma:U\to\mathbb{R}$ such that
\begin{equation*}
U\cap\partial\Omega=\{z\in U \mid \gamma(z)=0\},\quad \nabla\gamma(y_0)=\nu(y_0)\neq0,
\end{equation*}
where $\nu(y_0)$ is the unit outward normal at $y_0$. By continuity, we can find an open neighborhood $V\subset\Omega$ of $\xi_0$ and $\delta>0$ such that for every $\xi\in V$ and every $t\in(t_0-\delta,t_0+\delta)$, we have $\xi+t\omega\in U$.  
Define the $C^2$ map
\begin{equation*}
\Phi(\xi,t):=\gamma(\xi+t\omega),\quad (\xi,t)\in V\times(t_0-\delta,t_0+\delta).
\end{equation*}
We have $\Phi(\xi_0,t_0)=0$ and
\begin{equation*}
\frac{\partial\Phi}{\partial t}(\xi_0,t_0)=\nabla\gamma(y_0)\cdot\omega=\nu(y_0)\cdot\omega \neq 0,
\end{equation*}
by Lemma \ref{lem:nu_omega_bound}. The implicit function theorem (see \cite[Theorem 7.13.1]{Ciarlet2013}) provides an open neighborhood $W\subset\Omega$ of $\xi_0$ and a unique $C^2$ function $\tilde{\rho}:W\to(0,\infty)$ such that 
\begin{equation*}
\tilde{\rho}(\xi_0)=t_0 \quad \text{and} \quad \Phi(\xi,\tilde{\rho}(\xi))=0 \ \text{for all } \xi\in W.
\end{equation*}
By definition of $\rho(\xi,\omega)$ and the uniqueness of $\tilde{\rho}$, we must have 
\begin{equation*}
\tilde{\rho}(\xi)=\rho(\xi,\omega) \quad \text{on } W.
\end{equation*} 
Thus $\rho(\cdot,\omega)$ is $C^2$ in $W$. Since $\xi_0\in\Omega$ was arbitrary, $\rho(\cdot,\omega)$ is $C^2$ on the whole $\Omega$.

Now differentiate the identity $\gamma(\xi+\rho(\xi,\omega)\omega)=0$ with respect to $\xi$. Using the chain rule, we obtain
\begin{equation*}
\nabla\gamma(y)\cdot\left(I+\omega\otimes\nabla_\xi\rho(\xi,\omega)\right)=0,
\end{equation*}
where $y=\xi+\rho(\xi,\omega)\omega$. Since $\nabla\gamma(y)=\nu(y)$ and $\nu(y)\cdot(\omega\otimes\nabla_\xi\rho)=(\nu(y)\cdot\omega)\nabla_\xi\rho$, this yields
\begin{equation*}
\nu(y)+(\nu(y)\cdot\omega)\nabla_\xi\rho(\xi,\omega)=0.
\end{equation*}
Because $\nu(y)\cdot\omega>0$ by Lemma \ref{lem:nu_omega_bound}, we can solve for the gradient,
\begin{equation} \label{eq:gradient}
\nabla_\xi\rho(\xi,\omega)=-\frac{\nu(y)}{\nu(y)\cdot\omega}.
\end{equation}

Let $K\subset\Omega$ be any compact set. Because $\rho$ is continuous (Lemma \ref{lem:rho_continuous}) and $K\times\mathbb{S}^{n-1}$ is compact, there exist two constants $0<m_K<M_K$ such that
\begin{equation*}
m_K\le\rho(\xi,\omega)\le M_K\quad\forall(\xi,\omega)\in K\times\mathbb{S}^{n-1}.
\end{equation*}
Moreover, Lemma \ref{lem:nu_omega_bound} gives a constant $c_K>0$ with
\begin{equation*}
\nu(y)\cdot\omega\ge c_K\quad\forall(\xi,\omega)\in K\times\mathbb{S}^{n-1},
\end{equation*}
where $y=\xi+\rho(\xi,\omega)\omega$. From the gradient formula \eqref{eq:gradient}, we immediately obtain
\begin{equation*}
|\nabla_\xi\rho(\xi,\omega)|\le\frac{1}{c_K}\quad\forall(\xi,\omega)\in K\times\mathbb{S}^{n-1}.
\end{equation*}

To bound the second derivatives, we differentiate both sides of \eqref{eq:gradient} with respect to $\xi$ to obtain 
\begin{equation}
\nabla_\xi^2\rho = -\frac{1}{\nu\cdot\omega} \nabla_\xi\nu + \frac{1}{(\nu\cdot\omega)^2} \nu\otimes\nabla_\xi(\nu\cdot\omega). \label{eq:hessian_final}
\end{equation}
Since $\nu$ depends on $\xi$ through $y$, the chain rule gives
\begin{equation*}
\nabla_\xi\nu = (\nabla_y\nu) \nabla_\xi y,
\end{equation*}
where $\nabla_y\nu$ is the derivative of $\nu$ with respect to $y$, and $\nabla_\xi y = I + \omega\otimes\nabla_\xi\rho$ is the derivative of $y = \xi + \rho\omega$ with respect to $\xi$. Similarly, we have
\begin{equation*}
\nabla_\xi(\nu\cdot\omega) = \left(\nabla_\xi\nu\right)^\top \omega.
\end{equation*}

All quantities on the right‑hand side of \eqref{eq:hessian_final} are continuous on compact subsets of $\Omega\times\mathbb{S}^{n-1}$ because $\partial\Omega$ is $C^2$ and $\rho$ is continuous. Moreover, Lemma \ref{lem:nu_omega_bound} provides a uniform positive lower bound $c_K>0$ for $\nu\cdot\omega$ on any compact set $K\subset\Omega$. Consequently, we conclude that there exists a constant $D_K>0$ such that
\begin{equation*}
|\nabla_\xi^2\rho(\xi,\omega)|\le D_K \quad \forall(\xi,\omega)\in K\times\mathbb{S}^{n-1}.
\end{equation*}
The proof is completed by taking $C_K = \max\{1/c_K, D_K\}$.
\end{proof}

\begin{remark}\label{rem:gamma_independence}
In the derivation of the gradient formula \eqref{eq:gradient}, we introduced a local defining function $\gamma$ for the boundary $\partial\Omega$. The explicit form of $\gamma$ is irrelevant to the final result. Indeed, if one chooses a different defining function $\tilde{\gamma}$ with $\nabla\tilde{\gamma}(y)$ pointing either outward or inward, the same algebraic manipulation leads again to \eqref{eq:gradient}. 
\end{remark}

\begin{remark}\label{rem:regularity_transfer}
The regularity of $\rho(\cdot,\omega)$ is inherited directly from that of $\partial\Omega$ via the implicit function theorem. More precisely, if $\partial\Omega$ is of class $C^k$ with $k\ge 1$, then the function $\xi\mapsto\rho(\xi,\omega)$ is of class $C^k$ in a neighbourhood of $\xi$. This follows from the standard implicit function theorem applied to the defining equation $\gamma(\xi+\rho\omega)=0$, where $\gamma$ is a local $C^k$ defining function for $\partial\Omega$. 
\end{remark}

Combining these facts, we obtain the following result, which answers the open problem raised in \cite{ClappPistoiaSaldana2026} for $C^2$ bounded convex domain.

\begin{theorem} 
Let $\Omega\subset\mathbb{R}^n$ be a bounded convex domain with $C^2$ boundary. Then the function
\begin{equation*}
\psi_\Omega(\xi)=\int_{\mathbb{R}^n\setminus\Omega}\frac{\mathrm{d}x}{|x-\xi|^{2n}},\quad \xi\in\Omega,
\end{equation*}
has exactly one critical point in $\Omega$. This critical point is the global minimizer of $\psi_\Omega$.
\end{theorem}

\begin{proof}
We work with $F(\xi)=n\psi_\Omega(\xi)=\int_{\mathbb{S}^{n-1}}\rho(\xi,\omega)^{-n} \mathrm{d}\omega$. By Lemma \ref{lem:rho_C2}, for each fixed $\omega$ the map $\xi\mapsto\rho(\xi,\omega)$ is $C^2$ on $\Omega$, and its first and second derivatives are uniformly bounded on any compact set $K\subset\Omega$ independently of $\omega$. The dominated convergence theorem yields
\begin{equation} \label{eq:first}
\nabla F(\xi)= -n\int_{\mathbb{S}^{n-1}}\rho(\xi,\omega)^{-n-1}\nabla_\xi\rho(\xi,\omega) \mathrm{d}\omega,\quad \xi\in \Omega, 
\end{equation}
and 
\begin{equation} \label{eq:second}
\nabla^2 F(\xi)=n(n+1)\int_{\mathbb{S}^{n-1}}\rho^{-n-2} \nabla_\xi\rho\otimes\nabla_\xi\rho \mathrm{d}\omega
+n\int_{\mathbb{S}^{n-1}}\rho^{-n-1} (-\nabla_\xi^2\rho) \mathrm{d}\omega. 
\end{equation} 
The first integrand in \eqref{eq:second} is a positive semidefinite matrix multiplied by a positive coefficient. The second integrand in \eqref{eq:second} is also positive semidefinite because $\rho$ is concave (Lemma \ref{lem:concave}). Hence $\nabla^2 F(\xi)$ is positive semidefinite for every $\xi\in\Omega$.

To prove strict convexity, we fix an arbitrary nonzero vector $v\in\mathbb{R}^n$ and show that $v^\top\nabla^2F(\xi)v>0$. From \eqref{eq:second},
\begin{equation*}
v^\top\nabla^2F(\xi)v = n(n+1)\int_{\mathbb{S}^{n-1}}\rho^{-n-2} (v\cdot\nabla_\xi\rho)^2 \mathrm{d}\omega
+ n\int_{\mathbb{S}^{n-1}}\rho^{-n-1} v^\top(-\nabla_\xi^2\rho)v \mathrm{d}\omega.
\end{equation*}
If the first term were zero, then $v\cdot\nabla_\xi\rho(\xi,\omega)=0$ for every every $\omega$ by continuity. Using the gradient formula \eqref{eq:first}, we have
\begin{equation*}
v\cdot\nabla_\xi\rho = -\frac{v\cdot\nu}{\nu\cdot\omega}.
\end{equation*}
Therefore, the condition $v\cdot\nabla_\xi\rho(\xi,\omega)=0$ for every $\omega$ is equivalent to
\begin{equation} \label{eq:ortho}
v\cdot\nu(y(\omega))=0 \quad \text{for every } \omega.
\end{equation}
For any fixed $\xi \in \Omega$, recall $y=\xi+\rho(\xi,\omega)\omega$. We claim that the map $\omega \mapsto \nu(y(\omega))$ is surjective onto $\mathbb{S}^{n-1}$. Indeed, given any unit vector $u$, by Supporting Hyperplane Theorem (see \cite[Theorem 1.3.2]{Schneider2014}), there exists a boundary point $y$ with outer normal $\nu(y)=u$. Choosing 
\begin{equation*}
\omega = \frac{y-\xi}{|y-\xi|},
\end{equation*} 
which is well defined because $\xi\in\Omega$, gives $y(\omega)=y$ and then $\nu(y(\omega))=u$. Using \eqref{eq:ortho}, we have $v=0$, a contradiction. Therefore, 
\begin{equation*}
v^\top\nabla^2F(\xi)v>0, \quad \text{for every nonzero $v$,}
\end{equation*}
that is, $\nabla^2F(\xi)$ is positive definite. Thus $F$ is strictly convex on $\Omega$.

A strictly convex function on a convex open set can have at most one critical point, and any such point must be its global minimum. Moreover, $\psi_\Omega(\xi)\to+\infty$ as $\xi\to\partial\Omega$, so $F$ attains a minimum in $\Omega$. By strict convexity this minimum is unique and is the only critical point. Consequently, $\psi_\Omega$ has exactly one critical point, which is its global minimizer. See more details in Section \ref{section-4}.
This completes the proof.
\end{proof}

\section{The computation of $A_k$} \label{section-com}

In this section, we present the detailed computation of the integral that is essential for the analysis in Proposition \ref{prop:annulus}. 
Let $n\ge2$ and $\mathbb{S}^{n-1}$ be the unit sphere in $\mathbb{R}^n$ with surface area
\begin{equation*}
|\mathbb{S}^{n-1}|=\frac{2\pi^{\frac{n}{2}}}{\Gamma(\frac{n}{2})}.
\end{equation*}
Fix a direction $e_1=(1,0,\cdots,0)\in\mathbb{S}^{n-1}$ and let $\theta$ be the angle between a variable point $\omega\in\mathbb{S}^{n-1}$ and $e_1$, that is, $\cos\theta=\langle\omega,e_1\rangle$. 
Recall that the Gegenbauer polynomials $C_k^{(\lambda)}(t)$ satisfy the generating function (see \cite[6.4.10]{AndrewsAskeyRoy1999})
\begin{equation*}
(1-2tr+r^2)^{-n}=\sum_{k=0}^\infty C_k^{(n)}(t)r^k,\quad |r|<1,
\end{equation*}
and define
\begin{equation*}
A_k:=\int_{\mathbb{S}^{n-1}}C_k^{(n)}(\cos\theta)\mathrm{d}\omega,\quad k=0,1,2,\cdots .
\end{equation*}

\begin{lemma} \label{lem:com}
Define for $m\ge 0$
\begin{equation*}
I_m:=\int_{-1}^1 C_{2m}^{(n)}(t)(1-t^2)^{\frac{n-3}{2}}\mathrm{d}t, \quad
J_m:=\int_{-1}^1 t C_{2m+1}^{(n)}(t)(1-t^2)^{\frac{n-3}{2}}\mathrm{d}t, \label{eq:def}
\end{equation*}
where $C_k^{(n)}(t)$ denotes the Gegenbauer polynomial with parameter $\lambda=n$.
Then  
\begin{equation}
I_m = \frac{\sqrt{\pi}\Gamma\left(\frac{n-1}{2}\right)}{\Gamma\left(\frac{n}{2}\right)}\frac{(n)_m}{m!}\frac{n+2m}{n}, \quad m\ge 0,  \label{eq:closed1}
\end{equation}
and
\begin{equation}
J_m = \frac{\sqrt{\pi}\Gamma\left(\frac{n-1}{2}\right)}{\Gamma\left(\frac{n}{2}\right)}
\frac{(n)_m}{m!}\frac{2(n+m)}{n}, \quad m\ge 0, \label{eq:closed2}
\end{equation}
where $(n)_m=n(n+1)\cdots(n+m-1)$.
\end{lemma}

\begin{proof}
The proof uses two standard identities for Gegenbauer polynomials.
First, the three-term recurrence (see \cite[6.4.16]{AndrewsAskeyRoy1999})
\begin{equation}
(k+1)C_{k+1}^{(n)}(t)=2(n+k)t C_k^{(n)}(t)-(2n+k-1)C_{k-1}^{(n)}(t). \label{eq:3term}
\end{equation}
Second, a differential identity (see \cite[8.939]{GradshteuRyzhik2007})
\begin{equation}
(1-t^2)\frac{\mathrm{d}}{\mathrm{d}t}C_k^{(n)}(t) = (2n+k-1)C_{k-1}^{(n)}(t)-k t C_k^{(n)}(t). \label{eq:diff}
\end{equation}

Set $k=2m+1$ in \eqref{eq:3term} to obtain
\begin{equation}
(2m+2)C_{2m+2}^{(n)}(t)=2(n+2m+1)t C_{2m+1}^{(n)}(t)-(2n+2m)C_{2m}^{(n)}(t). \label{eq:rec_spec}
\end{equation}
Multiplying \eqref{eq:rec_spec} by $w(t):=(1-t^2)^{\frac{n-3}{2}}$ and integrating over $[-1,1]$ gives
\begin{equation}
(2m+2)I_{m+1}=2(n+2m+1)J_m-(2n+2m)I_m. \label{eq:step1}
\end{equation}
Set $k=2m+1$ in \eqref{eq:diff} and rearrange to isolate $tC_{2m+1}^{(n)}(t)$,
\begin{equation}
(2m+1)t C_{2m+1}^{(n)}(t) = (2m+2n)C_{2m}^{(n)}(t) - (1-t^2)\frac{\mathrm{d}}{\mathrm{d}t}C_{2m+1}^{(n)}(t). \label{eq:tc}
\end{equation}
Multiply \eqref{eq:tc} by $w(t)$ and integrate,
\begin{equation*} \begin{aligned} 
(2m+1)J_m =& (2m+2n)I_m - \int_{-1}^1 (1-t^2)\frac{\mathrm{d}}{\mathrm{d}t}C_{2m+1}^{(n)}(t)w(t)\mathrm{d}t \\
=& (2m+2n)I_m +\int_{-1}^1 C_{2m+1}^{(n)}(t)\frac{\mathrm{d}}{\mathrm{d}t}\left[(1-t^2)^{\frac{n-1}{2}}\right]\mathrm{d}t \\
=& (2m+2n)I_m +\int_{-1}^1 C_{2m+1}^{(n)}(t)\left[-(n-1)t w(t)\right]\mathrm{d}t \\
=& (2m+2n)I_m -(n-1)J_m, 
\end{aligned} \end{equation*}
which simplifies to
\begin{equation}
J_m = \frac{2(m+n)}{2m+n}I_m. \label{eq:Ym}
\end{equation}

Insert \eqref{eq:Ym} into \eqref{eq:step1},
\begin{equation*} \begin{aligned}
(2m+2)I_{m+1}
=& 2(n+2m+1)\frac{2(m+n)}{2m+n}I_m - 2(n+m)I_m \\
=& 2I_m\frac{n+m}{2m+n}(n+2m+2). 
\end{aligned} \end{equation*} 
Together with 
\begin{equation*} \begin{aligned}
I_0=\int_{-1}^1(1-t^2)^{\frac{n-3}{2}}\mathrm{d}t=B\left(\frac12,\frac{n-1}{2}\right)
=\frac{\sqrt{\pi} \Gamma(\frac{n-1}{2})}{\Gamma(\frac{n}{2})},
\end{aligned} \end{equation*} 
we obtain
\begin{equation*}
I_m=I_0\prod_{j=0}^{m-1} \frac{n+j}{j+1} \cdot \prod_{j=0}^{m-1} \frac{n+2j+2}{n+2j}
=\frac{\sqrt{\pi} \Gamma(\frac{n-1}{2})}{\Gamma(\frac{n}{2})}\cdot\frac{(n)_m}{m!}\cdot\frac{n+2m}{n},
\end{equation*} 
and 
\begin{equation*}
J_m=\frac{2(m+n)}{2m+n}I_m
=\frac{\sqrt{\pi} \Gamma(\frac{n-1}{2})}{\Gamma(\frac{n}{2})}\cdot\frac{(n)_m}{m!}\cdot\frac{2(n+m)}{n}.
\end{equation*}
This establishes \eqref{eq:closed1} and \eqref{eq:closed2}, thereby completing the proof.
\end{proof}

The following theorem provides an expression for the integral of Gegenbauer polynomials over the unit sphere, which is crucial in the analysis of Proposition \ref{prop:annulus}.

\begin{theorem} \label{thm:com}
Define
\begin{equation*}
A_k:=\int_{\mathbb{S}^{n-1}}C_k^{(n)}(\cos\theta) \mathrm{d}\omega,\quad k=0,1,2,\cdots
\end{equation*}
Then
\begin{equation*}
A_k=
\begin{cases}
0, & k=2m+1,\\
 |\mathbb{S}^{n-1}| \frac{(n)_m}{m!}\left(1+\frac{2m}{n}\right), & k=2m,
\end{cases}
\end{equation*}
where $(n)_m=n(n+1)\cdots(n+m-1)$.
\end{theorem}

\begin{proof}
For $k=2m+1$, the polynomial $C_{2m+1}^{(n)}(t)$ is odd (see \cite[6.4.11]{AndrewsAskeyRoy1999}). Under the change of variable $\omega\mapsto-\omega$ on the sphere, the surface measure $\mathrm{d}\omega$ is invariant, while $\cos\theta$ changes sign. Hence,
\begin{equation*} \begin{aligned}
A_{2m+1}
=&\int_{\mathbb{S}^{n-1}}C_{2m+1}^{(n)}(-\cos\theta) \mathrm{d}\omega 
= -A_{2m+1},
\end{aligned} \end{equation*}
which forces $A_{2m+1}=0$.

For $k=2m$, $C_{2m}^{(n)}(t)$ is even. Using the formula (see \cite[9.6.4]{AndrewsAskeyRoy1999}),
\begin{equation*}
\int_{\mathbb{S}^{n-1}}f(\cos\theta) \mathrm{d}\omega
=|\mathbb{S}^{n-2}|\int_{-1}^1 f(t)(1-t^2)^{\frac{n-3}{2}} \mathrm{d}t. \label{eq:sphere_int}
\end{equation*}
Thus
\begin{equation}
A_{2m}=|\mathbb{S}^{n-2}| I_m,\quad 
I_m:=\int_{-1}^1 C_{2m}^{(n)}(t)(1-t^2)^{\frac{n-3}{2}} \mathrm{d}t. \label{eq:Im_def}
\end{equation} 
The integrals $I_m$ have been studied in Lemma \ref{lem:com} and
\begin{equation}
I_m = \frac{\sqrt{\pi} \Gamma\left(\frac{n-1}{2}\right)}{\Gamma\left(\frac{n}{2}\right)} 
 \frac{(n)_m}{m!} \frac{n+2m}{n}. \label{eq:Iclosed}
\end{equation} 
Substituting \eqref{eq:Iclosed} into \eqref{eq:Im_def} and using the expression for $|\mathbb{S}^{n-2}|$ gives
\begin{equation*} \begin{aligned}
A_{2m}
=& \frac{2\pi^{\frac{n-1}{2}}}{\Gamma(\frac{n-1}{2})}
   \cdot \frac{\sqrt{\pi}\Gamma(\frac{n-1}{2})}{\Gamma(\frac{n}{2})}
   \cdot \frac{(n)_m}{m!}\cdot\frac{n+2m}{n} \\
=& \frac{2\pi^{\frac{n}{2}}}{\Gamma(\frac{n}{2})}\cdot\frac{(n)_m}{m!}\left(1+\frac{2m}{n}\right) \\
=& |\mathbb{S}^{n-1}|\frac{(n)_m}{m!}\left(1+\frac{2m}{n}\right).
\end{aligned} \end{equation*}
This completes the proof for even $k$, and together with the odd case establishes the theorem.
\end{proof}

\begin{remark}
The orthogonality property 
\begin{equation*} \label{eq:orthogonality}
\int_{\mathbb{S}^{n-1}}C_k^{(\frac{n-2}{2})}(\cos\theta)\mathrm{d}\omega=0 \ (k\ge1), \quad \int_{\mathbb{S}^{n-1}}C_0^{(\frac{n-2}{2})}(\cos\theta)\mathrm{d}\omega=\omega_{n-1},
\end{equation*}
holds precisely when the Gegenbauer parameter matches the dimension of the sphere, that is, $\lambda = \frac{n-2}{2}$. 
In this case $C_k^{(\frac{n-2}{2})}(\cos\theta)$ is the zonal spherical harmonic of degree $k$ on $\mathbb{S}^{n-1}$, and its integral over the sphere vanishes for $k\ge 1$ (see \cite[(9.6.6)]{AndrewsAskeyRoy1999}).

For a general parameter $\lambda\neq\frac{n-2}{2}$, the function $C_k^{(\lambda)}(\cos\theta)$ is no longer a spherical harmonic, and the integral $\int_{\mathbb{S}^{n-1}} C_k^{(\lambda)}(\cos\theta)\mathrm{d}\omega$ does not vanish.
We note that the methods employed in Lemma \ref{lem:com} and Theorem \ref{thm:com} are sufficiently general to treat arbitrary $\lambda>0$.  
\end{remark}

\section*{Acknowledge}
The research of Peng was supported by National Key R\&D Program (No. 2023YFA1010002). The research of Liu was supported by the Fundamental Research Funds for the Central Universities (No. 2662025XXQD003).

\subsection*{Data Availability Statement} Data sharing is not applicable to this article as no datasets were generated or analyzed during the current study.

\subsection*{Declarations}

\subsection*{Conflict of interest} The authors have no conflict of interest to declare that is relevant to the content of this article.

\end{document}